\documentclass[12pt]{article}
\usepackage{latexsym, amssymb, amsthm}
\usepackage{dsfont}

\DeclareMathAlphabet\gothic{U}{euf}{m}{n}

\makeatletter
\def\eqnarray{\stepcounter{equation}\let\@currentlabel=\theequation
\global\@eqnswtrue
\tabskip\@centering\let\\=\@eqncr
$$\halign to \displaywidth\bgroup\hfil\global\@eqcnt\z@
  $\displaystyle\tabskip\z@{##}$&\global\@eqcnt\@ne
  \hfil$\displaystyle{{}##{}}$\hfil
  &\global\@eqcnt\tw@ $\displaystyle{##}$\hfil
  \tabskip\@centering&\llap{##}\tabskip\z@\cr}

\def\endeqnarray{\@@eqncr\egroup
      \global\advance\c@equation\m@ne$$\global\@ignoretrue}

\def\@yeqncr{\@ifnextchar [{\@xeqncr}{\@xeqncr[5pt]}}
\makeatother

\begin{document}
\bibliographystyle{tom}

\newtheorem{lemma}{Lemma}[section]
\newtheorem{thm}[lemma]{Theorem}
\newtheorem{cor}[lemma]{Corollary}
\newtheorem{prop}[lemma]{Proposition}

\theoremstyle{definition}

\newtheorem{remark}[lemma]{Remark}
\newtheorem{exam}[lemma]{Example}
\newtheorem{definition}[lemma]{Definition}

\newcommand{\gota}{\gothic{a}}
\newcommand{\gotb}{\gothic{b}}
\newcommand{\gotc}{\gothic{c}}
\newcommand{\gote}{\gothic{e}}
\newcommand{\gotf}{\gothic{f}}
\newcommand{\gotg}{\gothic{g}}
\newcommand{\gothh}{\gothic{h}}
\newcommand{\gotk}{\gothic{k}}
\newcommand{\gotm}{\gothic{m}}
\newcommand{\gotn}{\gothic{n}}
\newcommand{\gotp}{\gothic{p}}
\newcommand{\gotq}{\gothic{q}}
\newcommand{\gotr}{\gothic{r}}
\newcommand{\gots}{\gothic{s}}
\newcommand{\gott}{\gothic{t}}
\newcommand{\gotu}{\gothic{u}}
\newcommand{\gotv}{\gothic{v}}
\newcommand{\gotw}{\gothic{w}}
\newcommand{\gotz}{\gothic{z}}
\newcommand{\gotA}{\gothic{A}}
\newcommand{\gotB}{\gothic{B}}
\newcommand{\gotG}{\gothic{G}}
\newcommand{\gotL}{\gothic{L}}
\newcommand{\gotS}{\gothic{S}}
\newcommand{\gotT}{\gothic{T}}

\newcounter{teller}
\renewcommand{\theteller}{(\alph{teller})}
\newenvironment{tabel}{\begin{list}%
{\rm  (\alph{teller})\hfill}{\usecounter{teller} \leftmargin=1.1cm
\labelwidth=1.1cm \labelsep=0cm \parsep=0cm}
                      }{\end{list}}

\newcounter{tellerr}
\renewcommand{\thetellerr}{(\roman{tellerr})}
\newenvironment{tabeleq}{\begin{list}%
{\rm  (\roman{tellerr})\hfill}{\usecounter{tellerr} \leftmargin=1.1cm
\labelwidth=1.1cm \labelsep=0cm \parsep=0cm}
                         }{\end{list}}

\newcounter{tellerrr}
\renewcommand{\thetellerrr}{(\Roman{tellerrr})}
\newenvironment{tabelR}{\begin{list}%
{\rm  (\Roman{tellerrr})\hfill}{\usecounter{tellerrr} \leftmargin=1.1cm
\labelwidth=1.1cm \labelsep=0cm \parsep=0cm}
                         }{\end{list}}

\newcounter{proofstep}
\newcommand{\nextstep}{\refstepcounter{proofstep}\vertspace \par 
          \noindent{\bf Step \theproofstep} \hspace{5pt}}
\newcommand{\firststep}{\setcounter{proofstep}{0}\nextstep}

\newcommand{\Ni}{\mathds{N}}
\newcommand{\Qi}{\mathds{Q}}
\newcommand{\Ri}{\mathds{R}}
\newcommand{\Ci}{\mathds{C}}
\newcommand{\Ti}{\mathds{T}}
\newcommand{\Zi}{\mathds{Z}}
\newcommand{\Fi}{\mathds{F}}

\renewcommand{\proofname}{{\bf Proof}}

\newcommand{\vertspace}{\vskip10.0pt plus 4.0pt minus 6.0pt}

\newcommand{\simh}{{\stackrel{{\rm cap}}{\sim}}}
\newcommand{\ad}{{\mathop{\rm ad}}}
\newcommand{\Ad}{{\mathop{\rm Ad}}}
\newcommand{\alg}{{\mathop{\rm alg}}}
\newcommand{\clalg}{{\mathop{\overline{\rm alg}}}}
\newcommand{\Aut}{\mathop{\rm Aut}}
\newcommand{\arccot}{\mathop{\rm arccot}}
\newcommand{\capp}{{\mathop{\rm cap}}}
\newcommand{\rcapp}{{\mathop{\rm rcap}}}
\newcommand{\diam}{\mathop{\rm diam}}
\newcommand{\divv}{\mathop{\rm div}}
\newcommand{\dom}{\mathop{\rm dom}}
\newcommand{\codim}{\mathop{\rm codim}}
\newcommand{\RRe}{\mathop{\rm Re}}
\newcommand{\IIm}{\mathop{\rm Im}}
\newcommand{\tr}{{\mathop{\rm Tr \,}}}
\newcommand{\Tr}{{\mathop{\rm Tr \,}}}
\newcommand{\Vol}{{\mathop{\rm Vol}}}
\newcommand{\card}{{\mathop{\rm card}}}
\newcommand{\rank}{\mathop{\rm rank}}
\newcommand{\supp}{\mathop{\rm supp}}
\newcommand{\sgn}{\mathop{\rm sgn}}
\newcommand{\essinf}{\mathop{\rm ess\,inf}}
\newcommand{\esssup}{\mathop{\rm ess\,sup}}
\newcommand{\Int}{\mathop{\rm Int}}
\newcommand{\lcm}{\mathop{\rm lcm}}
\newcommand{\loc}{{\rm loc}}
\newcommand{\HS}{{\rm HS}}
\newcommand{\Trn}{{\rm Tr}}
\newcommand{\n}{{\rm N}}
\newcommand{\WOT}{{\rm WOT}}

\newcommand{\at}{@}

\newcommand{\mod}{\mathop{\rm mod}}
\newcommand{\spann}{\mathop{\rm span}}
\newcommand{\one}{\mathds{1}}

\hyphenation{groups}
\hyphenation{unitary}

\newcommand{\tfrac}[2]{{\textstyle \frac{#1}{#2}}}

\newcommand{\ca}{{\cal A}}
\newcommand{\cb}{{\cal B}}
\newcommand{\cc}{{\cal C}}
\newcommand{\cd}{{\cal D}}
\newcommand{\ce}{{\cal E}}
\newcommand{\cf}{{\cal F}}
\newcommand{\ch}{{\cal H}}
\newcommand{\chs}{{\cal HS}}
\newcommand{\ci}{{\cal I}}
\newcommand{\ck}{{\cal K}}
\newcommand{\cl}{{\cal L}}
\newcommand{\cm}{{\cal M}}
\newcommand{\cn}{{\cal N}}
\newcommand{\co}{{\cal O}}
\newcommand{\cp}{{\cal P}}
\newcommand{\cs}{{\cal S}}
\newcommand{\ct}{{\cal T}}
\newcommand{\cx}{{\cal X}}
\newcommand{\cy}{{\cal Y}}
\newcommand{\cz}{{\cal Z}}

\thispagestyle{empty}

\vspace*{1cm}
\begin{center}
{\Large\bf H\"older kernel estimates for Robin operators  \\[1mm]
and Dirichlet-to-Neumann operators} \\[10mm]

\large A.F.M. ter Elst and M.F. Wong

\end{center}

\vspace{5mm}

\begin{center}
{\bf Abstract}
\end{center}

\begin{list}{}{\leftmargin=1.7cm \rightmargin=1.7cm \listparindent=10mm 
   \parsep=0pt}
\item
Consider the elliptic operator
\[
A = - \sum_{k,l=1}^d \partial_k \, c_{kl} \, \partial_l
   + \sum_{k=1}^d a_k \, \partial_k
   - \sum_{k=1}^d \partial_k \, b_k
   + a_0
\]
on a bounded connected open set $\Omega \subset \Ri^d$ with Lipschitz 
boundary conditions,
where $c_{kl} \in L_\infty(\Omega,\Ri)$ and 
$a_k,b_k,a_0 \in L_\infty(\Omega,\Ci)$, subject to 
Robin boundary conditions $\partial_\nu u + \beta \, \Tr u = 0$, 
where $\beta \in L_\infty(\partial \Omega, \Ci)$ is complex valued.
Then we show that the kernel of the semigroup generated by $-A$
satisfies Gaussian estimates and H\"older Gaussian estimates.
If all coefficients and the function $\beta$ are real valued, 
then we prove Gaussian lower bounds.

Finally, if $\Omega$ is of class $C^{1+\kappa}$ with $\kappa > 0$,
$c_{kl} = c_{lk}$ is H\"older continuous, $a_k = b_k = 0$ 
and $a_0$ is real valued,
then we show that the kernel of the semigroup associated to the 
Dirichlet-to-Neumann operator corresponding to $A$ has H\"older 
Poisson bounds.
\end{list}

\vspace{1cm}

\vspace{5mm}

\vspace{15mm}

\noindent
{\bf Home institution:}    \\[3mm]
\begin{tabular}{@{}l}
 Department of Mathematics   \\
 University of Auckland   \\
 Private bag 92019  \\
 Auckland  \\
 New Zealand  
\end{tabular}

\newpage

\section{Introduction} \label{Srobin1}

Second-order strongly elliptic operators in divergence form with 
real measurable bounded coefficients, subject to Dirichlet, Neumann
or mixed boundary conditions, are well studied on a bounded Lipschitz
domain~$\Omega$.
The kernel of the associated semigroup satisfies Gaussian kernel bounds
\cite{Aro}, \cite{Dav2}, \cite{AE1}, \cite{Daners}.
For Dirichlet and Neumann boundary conditions
it was proved that the kernel is even H\"older continuous
with the appropriate H\"older Gaussian bounds \cite{AT4}.
Recently, also H\"older Gaussian kernel bounds have been proved for operators
with mixed boundary conditions \cite{ERe2}.
The situation is different if the operator has Robin boundary conditions
$\partial_\nu u + \beta \, u = 0$,
where $\beta \in L_\infty(\partial \Omega)$.
If $\beta \geq 0$, then Gaussian kernel bounds were proved in 
\cite{AE1} Theorem~4.9, with a differentiability condition on the 
first-order coefficients and in \cite{Daners} Theorem~6.1 without 
the differentiability condition.
The condition $\beta \geq 0$ is replaced by $\beta \in L_\infty(\partial \Omega,\Ri)$
in \cite{Daners8} Theorem~2.2 and Lemma~3.2.
No H\"older Gaussian kernel bounds are known if $\beta \neq 0$.
In this paper we show that the kernel has both Gaussian kernel bounds 
and H\"older Gaussian bounds for any $\beta \in L_\infty(\partial \Omega)$,
even complex valued.
Also the lower-order coefficients can be complex, but we still require
that the principal coefficients are real valued (and measurable, although they do 
not have to be symmetric).

The first main theorem of this paper is as follows.

\begin{thm} \label{trobin101}
Let $\Omega \subset \Ri^d$ be a bounded Lipschitz domain.
For all $k,l \in \{ 1,\ldots,d \} $ let 
$c_{kl} \colon \Omega \to \Ri$ and $a_k,b_k,a_0 \colon \Omega \to \Ci$ 
be bounded and measurable.
Let $\beta \colon \partial \Omega \to \Ci$ be bounded and measurable.
Suppose there exists a $\mu > 0$ such that 
$\RRe \sum_{k,l=1}^d c_{kl}(x) \, \xi_k \, \overline{\xi_l} \geq \mu \, |\xi|^2$
for all $x \in \Omega$ and $\xi \in \Ci^d$.
Define the sesquilinear form $\gota_\beta \colon W^{1,2}(\Omega) \times W^{1,2}(\Omega) \to \Ci$
by 
\begin{eqnarray*}
\gota_\beta(u,v)
& = & \int_\Omega \sum_{k,l=1}^d c_{kl} (\partial_k u) \, \overline{\partial_l v} 
   + \int_\Omega \sum_{k=1}^d a_k \, (\partial_k u) \, \overline v
   + \int_\Omega \sum_{k=1}^d b_k \, u \, \overline{\partial_k \, v}   \\* 
& & \hspace*{20mm} {}
+ \int_\Omega a_0 \, u \, \overline v
+ \int_{\partial \Omega} \beta \, (\Tr u) \, \overline{\Tr v}
.  
\end{eqnarray*}
Let $A$ be the operator associated with the closed sectorial form~$\gota_\beta$.
Then the semigroup generated by $-A$ has a kernel~$K$.
Moreover, for all $\tau > 0$ and $\tau' \in (0,1)$ there exist $\kappa \in (0,1)$
and $b,c,\omega > 0$ such that 
\[
|K_t(x,y)| 
\leq c \, t^{-d/2} \, e^{-b\frac{|x-y|^2}{t}} \, e^{\omega t}
\]
and
\[|K_t(x,y) - K_t(x',y')| 
\leq c\, t^{-d/2} \, \Big( \frac{|x-x'|+|y-y'|}{t^{1/2} + |x-y|} \Big)^{\kappa} 
       \, e^{-b\frac{|x-y|^2}{t}} \, e^{\omega t} 
\]
for all $x,x',y,y' \in \Omega$ and $t>0$ with 
$|x-x'|+|y-y'| \leq \tau \, t^{1/2} + \tau' \, |x-y|$.
\end{thm}

For the proof we use a modification of the technique of Auscher 
\cite{Aus1} to use Morrey and Campanato spaces
to deduce H\"older Gaussian kernel bounds.
Here we will use a pointwise version of Morrey and Campanato seminorms
as in \cite{ERe2}.
Estimates are obtained
separately for the interior and regions near to the boundary. 
Then the Gaussian bounds follow from a similar result in \cite{EO2}.
We also prove that the constants in Theorem~\ref{trobin101} can be chosen uniformly 
with respect to the ellipticity constant and the $L_\infty$-norm of the coefficients
and~$\beta$.

If the lower order coefficients are real valued and the function $\beta$ is real 
valued, then the kernel $K$ is real valued.
If in addition the operator is self-adjoint, then $K$ satisfies Gaussian lower bounds.

\begin{thm} \label{trobin103}
Adopt the notation and assumptions of Theorem~\ref{trobin101}.
Assume in addition that the $a_k$, $b_k$, $a_0$ and $\beta$ are real valued.
Moreover, assume that the operator $A$ is self-adjoint. 
Then there are $b,c,\omega > 0$ such that 
\[
K_t(x,y) 
\geq c \, t^{-d/2} \, e^{-b\frac{|x-y|^2}{t}} \, e^{-\omega t}
\]
for all $x,y \in \Omega$ and $t > 0$.
\end{thm}

These lower bounds have been proved before by \cite{ChoulliKayser}
for $C^{1,1}$-domains and the parabolic problem associated to operators
in non-divergence form, possibly non-autonomous, and with Neumann boundary 
conditions.

The main regularity proposition that is used in the proof of 
Theorem~\ref{trobin101} can also be used for the Dirichlet-to-Neumann operator~$\cn$, 
which we describe next.

Let $\Omega \subset \Ri^d$ be a bounded Lipschitz domain with boundary~$\Gamma$.
Let $\varphi \in L_2(\Gamma)$. 
Then we say that $\varphi \in D(\cn)$ if there exists a $u \in W^{1,2}(\Omega)$
such that $\Delta u = 0$ weakly on $\Omega$, $\Tr u = \varphi$ and with
normal derivative $\partial_\nu u \in L_2(\Gamma)$. 
Then $\cn \varphi = \partial_\nu u$.
The Dirichlet-to-Neumann operator~$\cn$ is a positive self-adjoint operator.
Let $T$ be the semigroup on $L_2(\Gamma)$ generated by $-\cn$.
If there exists a $\kappa > 0$ such that $\Omega$ is a $C^{1 + \kappa}$ domain,
then ter Elst--Ouhabaz \cite{EO6} proved that $T$ has a kernel 
satisfying Poisson bounds.
The last main theorem of this paper is that the kernel 
of $T$ is H\"older continuous and satisfies H\"older continuous Poisson bounds.

\begin{thm} \label{trobin102}
Suppose $d \geq 3$.
Let $\Omega \subset \Ri^d$ be a bounded domain with 
$C^{1 + \kappa}$-boundary for some $\kappa > 0$.
Let $K$ be the kernel of the semigroup on $L_2(\Gamma)$ generated by $-\cn$,
where $\cn$ is the Dirichlet-to-Neumann operator.
Then for all $\varepsilon \in (0,1)$ and $\tau > 0$ there exist $c,\kappa' > 0$ 
such that 
\[
|K_t(x,y) - K_t(x',y')|
\leq c \, (t \wedge 1)^{-(d-1)} \,
   \Big( \frac{|x-x'|+|y-y'|}{t} \Big)^{\kappa'} \,
    \frac{1}{\displaystyle \Big( 1 + \frac{|x-y|}{t} \Big)^{d-\varepsilon}}
\]
for all $x,y,x',y' \in \Gamma$ with $|x-x'| + |y-y'| \leq \tau \, t$.
\end{thm}

In Theorem~\ref{trobin408} we will prove an extension of Theorem~\ref{trobin102}, 
where the Laplacian is replaced by a pure second-order strongly elliptic
operator in divergence form with real symmetric H\"older continuous coefficients.
This theorem will be a consequence of the previously mentioned Poisson kernel 
bounds in \cite{EO6} and a new theorem, Theorem~\ref{trobin401}, which provides optimal bounds 
for the semigroup from $L_2(\Gamma)$ into the space of H\"older continuous 
functions on $\Gamma$. 
The latter theorem is valid even for operators in divergence form with 
real measurable principal coefficients and complex lower-order terms.

The outline of the paper is as follows.
In Section~\ref{Srobin2} we introduce the notation and classes of 
coefficients that we need in this paper.
In Section~\ref{Srobin3} H\"older continuity of the
semigroup will be proved, together with a uniform version of Theorem~\ref{trobin101}. 
For this we use pointwise versions of Morrey and Campanato spaces 
and a regularity proposition to obtain elliptic regularity.
For the convenience of the reader we repeat in Section~\ref{Srobin2}
the pointwise versions of Morrey and Campanato spaces as introduced in \cite{ERe2}.
The regularity proposition contains a new boundary term to handle the Robin operator.
The same regularity proposition is then also used in Section~\ref{Srobin4} to prove 
an extension of Theorem~\ref{trobin102}.
In Section~\ref{Srobin5} we prove the lower kernel bounds of Theorem~\ref{trobin103}.
In the proof we use that a bounded connected Lipschitz domain satisfies
the chain condition. 
We prove this fact in the appendix.

\section{Preliminaries} \label{Srobin2}

In this section we introduce the classes of operators that we 
use throughout this paper. 
Since our proofs also involve Morrey and Campanato seminorms, we include
those definitions as well.

Let $\Omega \subset \Ri^d$ be open bounded with Lipschitz boundary $\Gamma$.
Let $\mu, M>0$.
Define $\ce_p(\Omega, \mu, M)$ to be the set of all measurable 
$C \colon \Omega \to \Ri^{d\times d}$ such that $\|C(x)\| \leq M$ for all 
$x\in \Omega$ and satisfy the ellipticity condition
\[
\sum_{k,l = 1}^d c_{kl}(x) \, \xi_k \, \overline{ \xi_l} \geq \mu \, |\xi|^2
\]
for all $\xi \in \Ri^d$ and $x\in \Omega$.
Here $\|C(x)\|$ is the $\ell_2$-norm of $C(x)$ in $\Ci^d$.
Let $\ce_p(\Omega) = \bigcup_{\mu, M>0} \ce_p(\Omega, \mu, M)$.

Let $\ce (\Omega, \mu, M)$ be the set of all tuples $(C, a, b,a_0)$, where 
$C \in \ce_p(\Omega, \mu, M)$, $a,b \colon \Omega \to \Ci^d$ measurable and 
$a_0 \colon \Omega \to \Ci$ measurable with $\|a(x)\|, \|b(x)\|, |a_0(x)| \leq M$ 
for all $x\in \Omega$.
Let $\ce (\Omega ) = \bigcup_{\mu, M>0} \ce (\Omega, \mu, M)$.

For all $(C,a,b,a_0)\in \ce (\Omega)$ define the closed sectorial forms 
$\gota_p, \gota \colon W^{1,2}(\Omega ) \times W^{1,2}(\Omega) \to \Ci $ by
\[
\gota_p (u,v) 
= \int_\Omega \sum_{k,l=1}^d 
   c_{kl} \, (\partial_k u) \, \overline{\partial_l v}
\]
and 
\[
\gota (u,v) 
= \int_\Omega \sum_{k,l=1}^d 
   c_{kl} \, (\partial_k u) \, \overline{\partial_l v} 
   + \int_\Omega \sum_{k=1}^d a_k \, (\partial_ku) \, \overline v 
   + \int_\Omega \sum_{l = 1}^d b_l \, u \, \overline{\partial_l v} 
   + \int_\Omega a_0 \, u \, \overline v
.  \]
Let $\beta \colon \Gamma \to \Ci$ be bounded measurable.
Define the closed sectorial form 
$\gota_\beta \colon W^{1,2}(\Omega ) \times W^{1,2}(\Omega) \to \Ci $ by
\[
\gota_\beta (u,v) = \gota (u,v) + \int_{\Gamma} \beta \, (\Tr u) \, \overline{\Tr v}
\]
It will be clear from the context what are $C, a, b, a_0$ and $\beta$.
Let $A$ be the m-sectorial operator associated to $\gota_\beta$.
We denote by $S$ the semigroup generated by $-A$.
We next show that $A$ is an elliptic operator with Robin boundary conditions.
In order to describe the domain of $A$, we need the notion of a 
weak co-normal derivative.

Define the operator $\ca \colon W^{1,2}(\Omega) \to (W^{1,2}_0(\Omega))^*$
by 
\[
\langle \ca u,v \rangle_{(W^{1,2}_0(\Omega))^* \times W^{1,2}_0(\Omega)}
= \gota(u,v)
.  \]
Let $u \in W^{1,2}(\Omega)$ and suppose that $\ca u \in L_2(\Omega)$.
Let $\psi \in L_2(\Gamma)$.
Then we say that $\psi$ is a {\bf weak co-normal derivative of $u$}
if 
\[
\gota(u,v) - (\ca u,v)_{L_2(\Omega)} = (\psi, \Tr v)_{L_2(\Gamma)}
\]
for all $v \in W^{1,2}(\Omega)$.
Then $\psi$ is unique by the Stone--Weierstra\ss\ theorem and we 
write $\partial_\nu u = \psi$.
If $u$ and $\Omega$ are smooth enough, then 
$\partial_\nu u 
= \sum_{k,l=1}^d \nu_l \, c_{kl} \, \partial_k u 
   + \sum_{k=1}^d \nu_k \, b_k \, u$.

\begin{lemma} \label{lRobin210}
$\dom(A_\beta) = \{ u \in W^{1,2}(\Omega) : 
   \ca u \in L_2(\Omega) \mbox{ and } \partial_\nu u + \Tr u = 0 \} $.
If $u \in \dom(A_\beta)$, then $A_\beta u = \ca u$.
\end{lemma}
\begin{proof}
The easy proof is left to the reader.
\end{proof}

Let $\kappa \in (0,1)$.
The space $C^\kappa(\Omega)$ is the space of all H\"older continuous functions 
of order $\kappa$ on $\Omega$ with semi-norm
\[
|||u|||_{C^\kappa(\Omega)}
= \sup \{ \frac{|u(x) - u(y)|}{|x-y|^\kappa} : x,y \in \Omega, \; 0 < |x-y| \leq 1 \}
.  \]
It is a Banach space with the norm 
$\|u\|_{C^\kappa(\Omega)} = |||u|||_{C^\kappa(\Omega)} + \|u\|_{L_\infty(\Omega)}$.

Finally we introduce the pointwise Morrey and Campanato semi-norms
as in \cite{ERe2}.
Let $\Omega \subset \Ri^d$ be open.
For all $x \in \Ri^d$ and $r > 0$ define $\Omega(x,r) = \Omega \cap B(x,r)$.
For all $\gamma \in [0,d]$, $R_e \in (0,1]$ and $x \in \Omega$ define 
$\|\cdot\|_{M,\gamma,x,\Omega,R_e} \colon L_2(\Omega) \to [0,\infty]$ by
\[
\|u\|_{M,\gamma,x,\Omega,R_e}
= \sup_{r \in (0,R_e]}
     \Big( r^{-\gamma} \int_{\Omega(x,r)} |u|^2 \Big)^{1/2}
 .  \]
Next, for all $\gamma \in [0,d+2]$, $R_e \in (0,1]$ and $x \in \Omega$ define 
$|||\cdot|||_{\cm,\gamma,x,\Omega,R_e} \colon L_2(\Omega) \to [0,\infty]$ by
\[
|||u|||_{\cm,\gamma,x,\Omega,R_e}
= \sup_{r \in (0,R_e]}
     \Big( r^{-\gamma} \int_{\Omega(x,r)} |u - \langle u\rangle_{\Omega(x,r)}|^2 \Big)^{1/2}
,  \]
where for an $L_2$-function $v$ we denote by 
$\langle v \rangle_{D} = \frac{1}{|D|} \int_D v$
the average of $v$ over a bounded measurable subset $D$ of the domain of $v$ with $|D| > 0$.

Since we consider a domain with a Lipschitz boundary, we need Lipschitz maps to a 
reference space, which we choose to be a cylinder.
Define
\[
E = \{x = (\tilde x,x_d) :  -1 < x_d < 1 \mbox{ and } \|\tilde x\|_{\Ri^{d-1}} < 1 \}
\]
the open cylinder in $\Ri^d$,
the lower half by $E^- = \{ x \in E  :  x_d < 0 \} $ and its mid plate by
\[
P = E \cap \{x \in \Ri^d : x_d=0 \}
.  \]

We emphasise that the field is the complex numbers and all our functions
are complex valued, except when explicitly stated otherwise.

\section{Robin boundary conditions} \label{Srobin3}

In this section we aim to prove a uniform version of Theorem~\ref{trobin101}.
The main tool is the following regularity result. 
For this section one may choose $\tilde \gamma = \gamma + \delta$. 
In Section~\ref{Srobin4} we will use the same proposition, but then the
choice $\tilde \gamma = \gamma + \delta$ does not work 
in the case $d=2$ and $d=3$.
In order not to repeat the major part of the proof, we prove a bit more 
in the next proposition.

\begin{prop} \label{probin301}
Let $\Omega \subset \Ri^d$ be open bounded with Lipschitz boundary $\Gamma$.
Let $U \subset \Ri^d$ be an open set and $\Phi$ a 
bi-Lipschitz map from an open neighbourhood of $\overline U$ onto 
an open subset of $\Ri^d$ such that $\Phi(U) = E$ and 
$\Phi(\Omega \cap U) = E^-$.
Then for all $\mu, M>0$ there exists a $\kappa \in (0,1)$ such that for all 
$\gamma, \tilde{\gamma} \in [0, d)$ and $\delta \in (0,2]$ with 
$\gamma + \delta < d-2+2\kappa$ and $\gamma + \delta \leq \tilde \gamma$ there exists a 
$c > 0$ such that the following is valid.
Let $C\in \ce_p (\Omega, \mu, M)$, $u,g \in W^{1,2}(\Omega)$, $\beta \in L_{\infty}(\Gamma)$ and 
$f, f_1, \ldots ,f_d \in L_2(\Omega)$.
Suppose that 
\begin{equation}
\gota_p(u,v)
= (f,v)_{L_2(\Omega)} + \sum_{i=1}^d (f_i, \partial_i v)_{L_2(\Omega)} 
   + \int_\Gamma \beta \, \Tr g \, \overline{\Tr v}
\label{eprobin301;1}
\end{equation}
for all $v \in W^{1,2}(\Omega)$.
Then 
\begin{eqnarray*}
\|\nabla(u \circ \Phi^{-1})\|_{M,\gamma+\delta, x, E^-, \frac{1}{2}}
& \leq & c \Big( \varepsilon^{2-\delta} \, \|f \circ \Phi^{-1}\|_{M,\gamma, x, E^-, \frac{1}{2}}
   + \sum_{i=1}^d \|f_i \circ \Phi^{-1}\|_{M,\gamma +\delta, x, E^-, \frac{1}{2}}
\\*
& & {}+ \varepsilon^{-(\gamma + \delta)} \, \|\nabla u\|_{L_2(\Omega)}
   + \varepsilon^{2-\delta} \, \|\beta\|_{L_{\infty}(\Gamma)}
            \, \|\nabla(g \circ \Phi^{-1})\|_{M,\gamma, x, E^-, \frac{1}{2}}
\\*
& &{}+ \varepsilon^{\tilde \gamma - \gamma - \delta} \, \|\beta\|_{L_{\infty}(\Gamma)}
    \, \|g \circ \Phi^{-1}\|_{M,\tilde \gamma, x, E^-, \frac{1}{2}}
           \Big)
\end{eqnarray*}
for all $x \in \frac{1}{2} E^-$ and $\varepsilon \in (0,1]$.
\end{prop}
\begin{proof}
The proof is a modification of the proof of \cite{ERe2} Proposition~6.5.
We indicate the differences and use the notation as in 
\cite{ERe2}.

Let $K \in [1,\infty)$ be larger than the Lipschitz constant 
of $\Phi|_{\Omega \cap U}$ and $\Phi^{-1}|_{E^-}$.
The trace map is continuous from $W^{1,1}(\Omega)$ into $L_1(\Gamma)$ by 
\cite{Nec2} Theorem~2.4.2.
Hence there exists a $c_1 \geq 1$ such that 
$\|\Tr v\|_{L_1(\Gamma)} \leq c_1 \, \|v\|_{W^{1,1}(\Omega)}$ for all 
$v \in W^{1,1}(\Omega)$.

After composition with $\Phi$ the equation (\ref{eprobin301;1}) 
transforms to an equation on $E^-$ with the aid of \cite{ERe2} Proposition~4.3
with a form with measurable coefficients $\tilde c_{kl}$.
For all $x \in \frac{1}{2} E^-$ and $0 < R \leq \frac{1}{2}$
define $P(x,R) = P \cap B(x,R)$.
Recall that $E^-(x,R) = E^- \cap B(x,R)$.
Let $W^{1,2}_{P(x,R)}(E^-(x,R))$ be the closure in $W^{1,2}(E^-(x,R))$
of the space
$ \{ w|_{E^-(x,R)} : w \in C_c^\infty(\Ri^d) \mbox{ and } 
         \supp w \cap (\partial (E^-(x,R)) \setminus P(x,R)) = \emptyset \} $.
By the De Giorgi estimates of \cite{ERe2} Lemma~5.1
there exist $\kappa \in (0,1)$ and $c_{DG} > 0$ 
such that 
\[
\int_{E^-(x,r)} |\nabla w|^2
\leq c_{DG} \, \Big( \frac{r}{R} \Big)^{d-2+2\kappa} \int_{E^-(x,R)} |\nabla w|^2 
\]
for all $x \in \frac{1}{2} E^-$, $r,R \in (0,1]$ and $w \in W^{1,2}(E^-(x,R))$
such that $r \leq R$ and 
\[
\sum_{k,l=1}^d \int_{E^-(x,R)}
\tilde c_{kl} \, (\partial_k w) \, \overline{\partial_l v} = 0
\]
for all $v \in W^{1,2}_{\Gamma(x,R)}(\Omega(x,R))$.

Let $x \in \frac{1}{2} E^-$ and $0 < R \leq \frac{1}{2}$.
By the Dirichlet-type Poincar\'e inequality of \cite{ERe2} Lemma~6.1(b)
and the Lax--Milgram theorem there exists a unique $\tilde v \in W^{1,2}_{P(x,R)}(E^-(x,R))$
such that 
\[
\sum_{k,l=1}^d \int_{E^-(x,R)} \tilde c_{kl} \, (\partial_k \tilde v) \, \overline{\partial_l \tau}
= \sum_{k,l=1}^d \int_{E^-(x,R)} 
     \tilde c_{kl} \, (\partial_k (u \circ \Phi^{-1})) \, \overline{\partial_l \tau}
\]
for all $\tau \in W^{1,2}_{P(x,R)}(E^-(x,R))$.
Define $v \colon \Omega \to \Ci$ by
\[
v(y)
= \left\{ \begin{array}{ll} 
     \tilde v(\Phi(y)) & \mbox{if } y \in \Phi^{-1}(E^-(x,R)) ,  \\[5pt]
     0 & \mbox{if } y \in \Omega \setminus \Phi^{-1}(E^-(x,R)) .
          \end{array} \right.
\]
Then $v \in W^{1,2}(\Omega)$ by in \cite{ERe2} Lemma~6.4.
Moreover, 
\begin{eqnarray*}
\sum_{k,l=1}^d \int_{E^-(x,R)} \tilde c_{kl} \, (\partial_k \tilde v) \, \overline{\partial_l \tilde v}
& = & \sum_{k,l=1}^d \int_{E^-(x,R)} 
   \tilde c_{kl} \, (\partial_k (u \circ \Phi^{-1})) \, \overline{\partial_l \tilde v}  \\
& = & \sum_{k,l=1}^d \int_{\Omega \cap U} 
    c_{kl} \, (\partial_k u) \, \overline{\partial_l v}  \\
& = & \gota_p(u,v)
= (f,v)_{L_2(\Omega)} + \sum_{i=1}^d (f_i, \partial_i v)_{L_2(\Omega)} 
    + \int_\Gamma \beta \, \Tr g \, \overline{\Tr v}
,  
\end{eqnarray*}
where the last term in the last step is new.
Using ellipticity and the Cauchy--Schwartz inequality, one obtains
\begin{eqnarray*}
\lefteqn{
(d! K^{d+2})^{-1} \, \mu \int_{E^-(x,R)} |\nabla \tilde v|^2
} \hspace*{8mm}   \\*
& \leq & d! \, K^d \Big( \int_{E^-(x,R)} |f \circ \Phi^{-1}|^2 \Big)^{1/2}
     \Big( \int_{E^-(x,R)} |\tilde v|^2 \Big)^{1/2} \\
& & {} + d! \, K^{d+1} \sum_{i=1}^d \Big( \int_{E^-(x,R)} |f_i \circ \Phi^{-1}|^2 \Big)^{1/2}
   \Big( \int_{E^-(x,R)} |\partial_i \tilde v|^2 \Big)^{1/2}
   + \Big| \int_\Gamma \beta \, \Tr g \, \overline{\Tr v} \Big|   \\
& \leq & 2 d! \, K^d \, \|f \circ \Phi^{-1}\|_{M,\gamma,x,E^-,\frac{1}{2}} \, R^{(\gamma+2)/2} \,
   \Big( \int_{E^-(x,R)} |\nabla \tilde v|^2 \Big)^{1/2} \\
& & {} + d! \, K^{d+1} \sum_{i=1}^d 
     \|f_i \circ \Phi^{-1} \|_{M, \gamma +\delta, x, E^-, \frac{1}{2}} 
     \, R^{(\gamma + \delta)/2} \, \Big( \int_{E^-(x,R)} |\nabla \tilde v|^2 \Big)^{1/2}
   + \Big| \int_\Gamma \beta \, \Tr g \, \overline{\Tr v} \Big|
,
\end{eqnarray*}
where we used the Dirichlet-type Poincar\'e inequality 
of \cite{ERe2} Lemma~6.1(b) in the last step.
We next estimate the boundary integral.
The boundedness of the trace gives
\begin{eqnarray*}
\lefteqn{
\Big| \int_\Gamma \beta \, \Tr g \, \overline{\Tr v} \Big|
} \hspace*{2mm}  \\*
& \leq & c_1 \, \|\beta\|_{L_{\infty}(\Gamma)} \|g \, \overline v\|_{W^{1,1}(\Omega)}  \\
& \leq & c_1 \, \|\beta\|_{L_{\infty}(\Gamma)} \int_\Omega 
    \Big( |g| \, |v| + |\nabla g| \, |v| + |g| \, |\nabla v| \Big)  \\
& \leq & c_1 \, d! \, d \, K^{d+1} \, \|\beta\|_{L_{\infty}(\Gamma)} \int_{E^-(x,R)} 
     \Big( |g \circ \Phi^{-1}| \, |\tilde v| 
   + |\nabla (g \circ \Phi^{-1})| \, |\tilde v| 
   + |g \circ \Phi^{-1}| \, |\nabla \tilde v| \Big)  \\
& \leq & c_1 \, d! \, d \, K^{d+1} \, \|\beta\|_{L_{\infty}(\Gamma)}
      \Big( R^{\tilde \gamma/2} \, \|g \circ \Phi^{-1}\|_{M,\tilde \gamma, x, E^-, \frac{1}{2}} 
   \cdot 2R   \\*
& & \hspace*{45mm} {}
   + R^{\gamma/2} \, \|\nabla(g \circ \Phi^{-1})\|_{M,\gamma, x, E^-, \frac{1}{2}} \cdot 2R
        \\*
& & \hspace*{60mm} {}
   + R^{\tilde \gamma/2} \, \|g \circ \Phi^{-1}\|_{M,\tilde \gamma, x, E^-, \frac{1}{2}}
       \Big)
       \Big( \int_{E^-(x,R)} |\nabla \tilde v|^2 \Big)^{1/2}  \\
& \leq & c_1 \, d! \, d \, K^{d+1} \, \|\beta\|_{L_{\infty}(\Gamma)}
   \Big( 3 R^{( \tilde \gamma - \gamma - \delta)/2 } 
       \, \|g \circ \Phi^{-1}\|_{M,\tilde \gamma, x, E^-, \frac{1}{2}} 
       \\*
& & \hspace*{36mm} {}
+ 2 R^{(2-\delta)/2} \, \|\nabla(g \circ \Phi^{-1})\|_{M,\gamma, x, E^-, \frac{1}{2}} 
    \Big) 
\cdot R^{(\gamma+\delta)/2} 
    \Big( \int_{E^-(x,R)} |\nabla \tilde v|^2 \Big)^{1/2} 
,
\end{eqnarray*}
where we used again the Dirichlet-type Poincar\'e inequality of \cite{ERe2} Lemma~6.1(b).
So 
\begin{eqnarray*}
\int_{E^-(x,R)} |\nabla \tilde v|^2
& \leq & (3 c_1 \, d!^2 \, d \, K^{2d+3})^2 \, \mu^{-2} 
    \Big( R^{(2-\delta)/2} \, \|f \circ \Phi^{-1}\|_{M,\gamma,x,E^-,\frac{1}{2}}
\\*
& & \hspace*{40mm} {} 
+ \sum_{i=1}^d \|f_i \circ \Phi^{-1} \|_{M, \gamma +\delta, x, E^-, \frac{1}{2}}
\\*
& & \hspace*{40mm} {}
+ R^{( \tilde \gamma - \gamma - \delta)/2 } \|\beta\|_{L_{\infty}(\Gamma)}
    \, \|g \circ \Phi^{-1}\|_{M,\tilde \gamma, x, E^-, \frac{1}{2}} 
\\*
& & \hspace*{40mm} {}
+ R^{(2-\delta)/2} \|\beta\|_{L_{\infty}(\Gamma)} \, 
       \|\nabla(g \circ \Phi^{-1})\|_{M,\gamma, x, E^-, \frac{1}{2}} 
    \Big)^2 R^{\gamma + \delta}
.
\end{eqnarray*}
Next let $r \in (0,R]$.
We apply the De Giorgi estimates to the function $w = u - v$.
Then 
\begin{eqnarray*}
\int_{E^-(x,r)} |\nabla (u \circ \Phi^{-1})|^2
& \leq & 2 \int_{E^-(x,r)} |\nabla (w \circ \Phi^{-1})|^2 
   + 2 \int_{E^-(x,r)} |\nabla \tilde v|^2 \nonumber \\
& \leq & 2 c_{DG} \Big( \frac{r}{R} \Big)^{d-2+2\kappa} \int_{E^-(x,R)} |\nabla (w \circ \Phi^{-1})|^2
   + 2 \int_{E^-(x,r)} |\nabla \tilde v|^2 \nonumber  \\
& \leq & 4 c_{DG} \Big( \frac{r}{R} \Big)^{d-2+2\kappa} \int_{E^-(x,R)} |\nabla (u \circ \Phi^{-1})|^2
   + (2 + 4 c_{DG}) \int_{E^-(x,R)} |\nabla \tilde v|^2    \\
& \leq & 4 c_{DG} \Big( \frac{r}{R} \Big)^{d-2+2\kappa} \int_{E^-(x,R)} |\nabla (u \circ \Phi^{-1})|^2
\\*
& & \hspace*{1mm} {}
   + c_2 \Big( R^{(2-\delta)/2} \, \|f \circ \Phi^{-1}\|_{M,\gamma,x,E^-,\frac{1}{2}}
   + \sum_{i=1}^d \|f_i \circ \Phi^{-1} \|_{M, \gamma +\delta, x, E^-, \frac{1}{2}}
      \\*
& & \hspace*{15mm} {}
   + R^{( \tilde \gamma - \gamma - \delta)/2 } \|\beta\|_{L_{\infty}(\Gamma)}
   \, \|g \circ \Phi^{-1}\|_{M,\tilde \gamma, x, E^-, \frac{1}{2}} 
      \\*
& & \hspace*{15mm} {}
   + R^{(2-\delta)/2} \|\beta\|_{L_{\infty}(\Gamma)} \, 
      \|\nabla(g \circ \Phi^{-1})\|_{M,\gamma, x, E^-, \frac{1}{2}} 
      \Big)^2 R^{\gamma + \delta}
,
\end{eqnarray*}
where $c_2 = (2 + 4 c_{DG}) (3 c_1 \, d! \, d \, K^{2d+3} \, M)^2 \, \mu^{-2} $.
Note that these bounds are uniform for all $x \in \frac{1}{2} E^-$ and
$0 < r \leq R \leq \frac{1}{2}$.
Moreover, $\gamma + \delta < d - 2 + 2 \kappa$.
Hence they can be improved by use of Lemma III.2.1 of \cite{Gia1} and one deduces that 
there exists a $c_3 > 0$, depending only on $c_{DG}$, $\gamma+\delta$ and $d-2+2\kappa$,
such that 
\begin{eqnarray*}
\int_{E^-(x,r)} |\nabla (u \circ \Phi^{-1})|^2
& \leq & c_3 \Big( \frac{r}{R} \Big)^{\gamma + \delta} \int_{E^-(x,R)} |\nabla (u \circ \Phi^{-1})|^2
\\*
& & \hspace*{10mm} {}
+ c_2 \, c_3
   \Big( \varepsilon^{2-\delta} \, \|f \circ \Phi^{-1}\|_{M,\gamma,x,E^-,\frac{1}{2}}
+ \sum_{i=1}^d \|f_i \circ \Phi^{-1} \|_{M, \gamma +\delta, x, E^-, \frac{1}{2}}
\\*
& & \hspace*{30mm} {}
+ \varepsilon^{\tilde \gamma - \gamma - \delta } \, \|\beta\|_{L_{\infty}(\Gamma)}
     \, \|g \circ \Phi^{-1}\|_{M,\tilde \gamma, x, E^-, \frac{1}{2}} 
\\*
& & \hspace*{30mm} {}
+ \varepsilon^{2-\delta} \, \|\beta\|_{L_{\infty}(\Gamma)} \, 
       \|\nabla(g \circ \Phi^{-1})\|_{M,\gamma, x, E^-, \frac{1}{2}} 
    \Big)^2 r^{\gamma + \delta}
,
\end{eqnarray*}
uniformly for all $x \in \frac{1}{2} E^-$, $\varepsilon \in (0,1]$
and $0 < r \leq R \leq \frac{1}{2} \, \varepsilon^2$.
Choosing $R = \frac{1}{2} \, \varepsilon^2$ gives
\begin{eqnarray*}
\int_{E^-(x,r)} |\nabla (u \circ \Phi^{-1})|^2
& \leq & 2^{\gamma + \delta} c_3 \,
(\varepsilon^{-(\gamma + \delta)} \, \|\nabla (u \circ \Phi^{-1})\|_{L_2(E^-)})^2 
   \, r^{\gamma + \delta}
\\*
& & \hspace*{10mm} {}
+ c_2 \, c_3
   \Big( \varepsilon^{2-\delta} \, \|f \circ \Phi^{-1}\|_{M,\gamma,x,E^-,\frac{1}{2}}
   + \sum_{i=1}^d \|f_i \circ \Phi^{-1} \|_{M, \gamma +\delta, x, E^-, \frac{1}{2}}
\\*
& & \hspace*{30mm} {}
+ \varepsilon^{\tilde \gamma - \gamma - \delta } \|\beta\|_{L_{\infty}(\Gamma)}
    \, \|g \circ \Phi^{-1}\|_{M,\tilde \gamma, x, E^-, \frac{1}{2}} 
\\*
& & \hspace*{30mm} {}
+ \varepsilon^{2-\delta} \|\beta\|_{L_{\infty}(\Gamma)} \, 
      \|\nabla(g \circ \Phi^{-1})\|_{M,\gamma, x, E^-, \frac{1}{2}} 
   \Big)^2 r^{\gamma + \delta}
,
\end{eqnarray*}
for all $x \in \frac{1}{2} E^-$ and $0 < r \leq \frac{1}{2} \, \varepsilon^2$.

The rest of the proof is similarly to the proof of \cite{ERe2} Proposition~6.5,
which is a modification of the proof of Proposition~3.2 in \cite{ERe2}.
\end{proof}

We also need the Davies perturbation.
Let 
\[
\cd = \{ \psi \in C^\infty_{\rm c}(\Ri^d,\Ri) : \|\nabla \psi\|_\infty \leq 1 \} 
.  \]
For all $\rho \in \Ri$ and $\psi \in \cd$ define the 
multiplication operator $U_\rho$ by $U_\rho u = e^{-\rho \, \psi} u$.
Note that $U_\rho u \in W^{1,2}(\Omega)$ for all 
$u \in W^{1,2}(\Omega)$.
Let $S^\rho_t = U_\rho \, S_t \, U_{-\rho}$ be the Davies perturbation
for all $t > 0$.
Let $- A^{(\rho)}$ the generator of $S^\rho$. 
Then $A^{(\rho)}$ is the operator associated with the form
$\gota^{(\rho)}_\beta$ with form domain $D(\gota^{(\rho)}_\beta) = W^{1,2}(\Omega)$ and  
\begin{equation}
\gota^{(\rho)}_\beta(u,v)
= \gota_p(u,v)
   + \int_\Omega \sum_{k=1}^d \Big( a^{(\rho)}_k \, (\partial_k u) \, \overline v
          + b^{(\rho)}_k \, u \, \overline{(\partial_k v)} \Big)
   + \int_\Omega a^{(\rho)}_0 \, u \, \overline v 
   + \int_{\Gamma} \beta \, (\Tr u) \, \overline{ \Tr v}
\label{erobin301;5}
\end{equation}
with 
\[
a^{(\rho)}_k = a_k - \rho \sum_{l =1}^d  c_{kl} \, \partial_l \psi
\quad , \quad
b^{(\rho)}_k = b_k + \rho \sum_{l =1}^d c_{lk} \, \partial_l \psi
\]
and 
\[
a^{(\rho)}_0
= a_0 - \rho^2 \sum_{k, l =1}^d c_{kl} \, (\partial_k \psi) \, \partial_l  \psi
      + \rho \sum_{k=1}^d a_k \, \partial_k \psi 
      - \rho \sum_{k=1}^d b_k \, \partial_k \psi  
.  \]

\begin{lemma} \label{lrobin302}
Let $\Omega\subset \Ri^d$ be open bounded with Lipschitz boundary $\Gamma$.
Then for all $\mu, M>0$, there exist $c_0, \omega_0 > 0$ such that
\begin{eqnarray*}
\|S^\rho_t u\|_{L_2(\Omega)} 
& \leq & e^{\omega_0 (1+\rho^2) t} \, \|u\|_{L_2(\Omega)} \nonumber \\
\|\nabla S^\rho_t u\|_{L_2(\Omega)} 
& \leq & c_0 \, t^{-1/2}e^{\omega_0 (1+\rho^2)  t} \, \|u\|_{L_2(\Omega)}  \\
\|A^{(\rho)} S^\rho_t u\|_{L_2(\Omega)} 
& \leq & c_0 \, t^{-1} \, e^{\omega_0 (1+\rho^2) t} \, \|u\|_{L_2(\Omega)}  \nonumber
\end{eqnarray*}
for all $(C, a, b, a_0) \in \ce (\Omega, \mu, M)$, $\beta \in L_{\infty} (\Gamma)$, 
$t>0$, $\rho \in \Ri$ and $\psi \in \cd$ with $\|\beta\|_{L_{\infty}(\Gamma)} \leq M$, 
where $S^\rho$ is the semigroup generated by $-A^{(\rho)}$.
\end{lemma}
\begin{proof}
Without lost of generality we may assume that $\mu \leq 1$.
By \cite{Nec2} Theorem~2.4.2 there exists a $c_1 > 0$ such that 
$\|\Tr v \|_{L_1 (\Gamma)} \leq c_1 \, \|v\|_{W^{1,1}(\Omega)}$ for all $v\in W^{1,1}(\Omega)$.
Let $u\in L_2(\Omega)$.
Then the boundary term can be estimated by
\begin{eqnarray}
|\int_{\Gamma} \beta \, (\Tr S_t^\rho u) \, \overline{\Tr S_t^\rho u}| 
& \leq & c_1 \, \|\beta\|_{L_{\infty}(\Gamma)} \, \|(S_t^\rho u) \, \overline{ S_t^\rho u}
         \, \|_{W^{1,1}(\Omega)} \nonumber \\
& \leq & c_2 \, (\|S_t^\rho u\|^2_{L_2(\Omega)}
     + 2 \|\nabla S_t^\rho u\|_{L_2(\Omega)} \, \|S_t^\rho u\|_{L_2(\Omega)}) , \nonumber
\end{eqnarray}
where $c_2 = c_1 \, M$.
Now by ellipticity
\begin{eqnarray}
\mu \, \|\nabla S_t^\rho u\|^2_{L_2(\Omega)} 
& \leq & \RRe \gota_p(S_t^\rho u) \nonumber \\
& \leq & \RRe \gota^{(\rho)}_\beta(S_t^\rho u) 
   + 2 d \, M \, (1+|\rho |) \, \|\nabla S_t^\rho u\|_{L_2(\Omega)} 
          \, \|S_t^\rho u\|_{L_2(\Omega)} \nonumber \\
& & \hspace{10mm} {}
   + M \, (1+|\rho|)^2 \, \|S_t^\rho u\|_{L_2(\Omega)}^2 
   + c_2 \, (\|S_t^\rho u\|^2 + 2 \|\nabla S_t^\rho u\|_{L_2(\Omega)} \, \|S_t^\rho u\|_{L_2(\Omega)}) \nonumber \\
& \leq & \RRe \gota^{(\rho)}_\beta(S_t^\rho u) 
   + 2 (1+|\rho|) (d \, M + c_2) \, \|\nabla S_t^\rho u\|_{L_2(\Omega)}
        \, \|S_t^\rho u\|_{L_2(\Omega)} \nonumber \\
& & \hspace{10mm} {} 
   + (1+|\rho|)^2 \, (M + c_2) \, \|S_t^\rho u\|_{L_2(\Omega)}^2 \nonumber \\
& \leq & \RRe \gota^{(\rho)}_\beta(S_t^\rho u) 
   + \frac{\mu}{2} \, \|\nabla S_t^\rho u\|^2_{L_2(\Omega)} 
   + \omega_1 \, (1+\rho^2) \, \|S_t^\rho u\|_{L_2(\Omega)}^2 \nonumber
\end{eqnarray}
for all $t>0$, where $\omega_1 = \frac{4}{\mu} \, (d \, M + c_2)^2 + 2(M +c_2)$.
Therefore
\[
\frac{1}{2} \, \mu \, \|\nabla S_t^\rho u\|^2_{L_2(\Omega)}  
\leq \RRe \gota^{(\rho)}_\beta(S_t^\rho u) 
   + \omega_1 \, (1+\rho^2) \, \|S_t^\rho u\|_{L_2(\Omega)}^2 
.  \]
Differentiating gives
\begin{eqnarray}
\frac{d}{dt} \, \|S_t^\rho u\|^2_{L_2(\Omega)} 
& = & -2 \RRe (A^{(\rho)} \, S_t^\rho u, S_t^\rho u)_{L_2(\Omega)} \nonumber \\
& = & -2 \RRe \gota^{(\rho)}_\beta(S_t^\rho u) 
\leq 2 \omega_1 \, (1+\rho^2) \, \|S_t^\rho u\|^2_{L_2(\Omega)} 
. \nonumber
\end{eqnarray}
Hence by Gronwall's lemma
\[
\|S_t^\rho u\|_{L_2(\Omega)} 
\leq e^{\omega_1 (1+\rho^2)t} \, \|u\|_{L_2(\Omega)} 
.\]
The estimates for $\|A^{(\rho)} \, S_t u\|_{L_2(\Omega)}$ and 
$\|\nabla S_tu\|_{L_2(\Omega)}$ follows from \cite{ERe2} Lemma~7.1.
\end{proof}

We next consider the $L_2 \to L_{\infty}$ and H\"older estimates for the semigroup near $\Gamma$.

\begin{prop} \label{probin303}
Let $\Omega \subset \Ri^d$ be open bounded with Lipschitz boundary $\Gamma$.
Let $U\subset \Ri^d$ and $\Phi$ be a bi-Lipschitz map from an open neighbourhood of 
$\overline U$ onto an open subset of $\Ri^d$ such that $\Phi(U) = E$ and $\Phi(\Omega \cap U) = E^-$.
Then for all $\mu, M>0$ there exist $\kappa \in (0,1)$ and $c, \omega >0$ such that the following is valid.
Let $(C,a,b,a_0)\in \ce (\Omega, \mu, M)$ and $\beta \in L_{\infty}(\Gamma)$ 
with $\|\beta\|_{L_{\infty}(\Gamma)} \leq M$.
Let $S$ be the semigroup generated by $-A$.
Then
\[
\|S_t^\rho u\|_{L_\infty(\Phi^{-1}(\frac{1}{2} \, E^-))}
\leq c \, t^{-d/4} \, e^{\omega (1+\rho^2) t} \, \|u\|_{L_2(\Omega)}
\]
and 
\[
|(S_t^\rho u)(x) - (S_t^\rho u)(y)|
\leq c \, t^{-d/4} \, t^{-\kappa/2} \, e^{\omega (1+\rho^2) t} \, \|u\|_{L_2(\Omega)} \, |x-y|^\kappa
\]
for all $t > 0$, $u \in L_2(\Omega)$, $\rho \in \Ri $, $\psi \in \cd$ 
and $x,y \in \Phi^{-1}(\frac{1}{2} \, E^-)$ with $|x-y| \leq \frac{1}{4K}$, 
where $K > 1$ is larger than the Lipschitz constant of $\Phi|_{\Omega \cap U}$ and $\Phi^{-1}|_{E^-}$.
\end{prop}
\begin{proof}
Let $\mu, M >0$ and let $\kappa \in (0,1)$ be as in Proposition \ref{probin301}.
For all $\gamma \in [0, d-2+2\kappa )$ let $P(\gamma)$ be the hypothesis
\begin{list}{}{\leftmargin=1.8cm \rightmargin=1.8cm \listparindent=12pt}
\item
There exist $c,\omega > 0$, depending only on $K$, $\mu$, $M$, $\kappa$ and 
$c_{DG}$, such that 
\begin{equation}
\|(S_t^\rho u) \circ \Phi^{-1}\|_{M,\gamma,x, E^-,\frac{1}{2}}
\leq c \, t^{-\gamma / 4} \, e^{\omega (1+\rho^2)t} \, \|u\|_{L_2(\Omega)}
\label{eprobin303;3}
\end{equation}
and 
\begin{equation}
\|\nabla ((S_t^\rho u) \circ \Phi^{-1})\|_{M,\gamma,x, E^-,\frac{1}{2}}
\leq c \, t^{-\gamma / 4} \, t^{-1/2} \, e^{\omega (1+\rho^2) t} \, \|u\|_{L_2(\Omega)}
\label{eprobin303;4}
\end{equation}
for all $t > 0$, $u \in L_2(\Omega)$, $\rho \in \Ri$, $\psi \in \cd$
and $x \in \frac{1}{2} \, E^-$.
\end{list}

Clearly $P(0)$ is valid by Lemma~\ref{lrobin302}.

\begin{lemma}\label{lrobin304} 
Let $\gamma \in [0,d-2 + 2\kappa)$ and suppose that $P(\gamma)$ is 
valid.
Let $\delta \in (0,2]$ and suppose that $\gamma + \delta < d-2+2\kappa$.
Then $P(\gamma+\delta)$ is valid.
\end{lemma}
\begin{proof}
Let $c_0,\omega_0 > 0$ be as in Lemma~\ref{lrobin302}.
Let $t > 0$, $u \in L_2(\Omega)$, $\rho \in \Ri$, $\psi \in \cd $ and $x \in \frac{1}{2} \, E^-$.
Note that
\begin{equation}
\|(S_t^\rho u) \circ \Phi^{-1}\|_{L_2(E^-)} 
\leq d! \, K^d \, \|S_t^\rho u\|_{L_2(\Omega)}
\leq d! \, K^d \, e^{\omega_0 (1+\rho^2) t} \, \|u\|_{L_2(\Omega)} ,
\label{elrobin304;1}
\end{equation}
by Lemma~\ref{lrobin302}.

We first prove the bounds (\ref{eprobin303;3}) for $P(\gamma+\delta)$.
Choose $\varepsilon = t^{1/4} \, e^{-t} \in (0,1]$.
Let $c_1$ be as in \cite{ERe2} Lemma 6.2.
Then it follows from 
\cite{ERe2} Lemma~6.2, (\ref{eprobin303;4}) and (\ref{elrobin304;1}) that 
\begin{eqnarray*}
\lefteqn{
\|(S_t^\rho u) \circ \Phi^{-1}\|_{\cm,\gamma+\delta,x, E^-,\frac{1}{2}}
} \hspace*{15mm} \\*
& \leq & c_1 ( \varepsilon^{2-\delta} \, \|\nabla ((S_t^\rho u) \circ \Phi^{-1})\|_{M,\gamma,x, E^-,\frac{1}{2}}
               + \varepsilon^{-(\gamma + \delta)} \, \|(S_t^\rho u) \circ \Phi^{-1}\|_{L_2(E^-)} ) \\
& \leq & c_1 \, ( \varepsilon^{2-\delta} \, c \, t^{-\gamma / 4} \, t^{-1/2} \, e^{\omega (1+\rho^2)  t}
   + \varepsilon^{-(\gamma + \delta)} \, d! \, K^d \, e^{\omega_0 (1+\rho^2) t} ) \, \|u\|_{L_2(\Omega)}  \\
& \leq & c' \, t^{-(\gamma + \delta) / 4} \, e^{\omega' (1+\rho^2) t} \, \|u\|_{L_2(\Omega)}
\end{eqnarray*}
where $c' = c_1 \, (c + d! \, K^d)$ and $\omega' = \omega_0 + \omega + \gamma + \delta$.
By \cite{ERe2} Lemma~3.1(a) there exist $c_2, c_3 > 0$ such that
\[
\|v\|_{M,\gamma+\delta,x, E^-,\frac{1}{2}}
\leq c_2 \, \|v\|_{\cm,\gamma+\delta,x, E^-,\frac{1}{2}}
   + c_3 \, \|v\|_{L_2(E^-)}
\]
for all $x \in \frac{1}{2} \, E^-$ and $v \in L_2(E^-)$.
Hence 
\begin{eqnarray}
\|(S_t^\rho u) \circ \Phi^{-1}\|_{M,\gamma+\delta,x, E^-,\frac{1}{2}}
& \leq & c_2 \, c' \, t^{-(\gamma + \delta) / 4} \, e^{\omega' (1+\rho^2) t} \, \|u\|_{L_2(\Omega)}
   + c_3 \, d! \, K^d \, e^{\omega_0 (1+\rho^2) t} \, \|u\|_{L_2(\Omega)}  \nonumber  \\
& \leq & c'' \, t^{-(\gamma + \delta) / 4} \, e^{\omega'' (1+\rho^2) t} \, \|u\|_{L_2(\Omega)}
,   \label{elrobin304;2}
\end{eqnarray}
where $c'' = c' \, c_2 + c_3 \, d! \, K^d$ and $\omega'' = \omega_0 + \omega' + d + 2$.
This gives the bound (\ref{eprobin303;3}) for $P(\gamma + \delta)$.

In order to obtain (\ref{eprobin303;4}), we use Proposition~\ref{probin301}.
Note that
\[ 
\gota_\beta^{(\rho)}(S_t^\rho u, v) 
= (S_{t/2}^\rho \, A^{(\rho)} \, S_{t/2}^\rho u, v)_{L_2(\Omega)}\]
for all $v \in W^{1,2}(\Omega)$.
It follows from (\ref{erobin301;5}) that
\[
\gota_p(S_t^\rho u, v) 
= (f, v)_{L_2(\Omega)} 
   - \sum_{i=1}^d (f_i, \partial_i v)_{L_2(\Omega)} 
   - \int_{\Gamma} \beta \, (\Tr S_t^\rho u) \, \overline{\Tr v}
\]
for all $v \in W^{1,2}(\Omega)$, where $f_i = b_i^{(\rho)} \, S_tu$ and
\[
f 
= S_{t/2}^\rho \, A^{(\rho)} \, S_{t/2}^\rho u 
   - a_0^{(\rho)} \, S_t^\rho u 
   - \sum_{i=1}^d a_i^{(\rho)} \, \partial_i \, S_t^\rho u 
.\]
Apply Proposition \ref{probin301} with $\varepsilon = t^{1/4} \, e^{-t} \in (0,1]$.
The three terms in $f$ are approximated separately using Lemma~\ref{lrobin302} with 
$\tilde{\gamma} = \gamma + \delta$.
First,
\begin{eqnarray*}
\lefteqn{
\varepsilon^{2-\delta} \, 
   \|(S_{t/2}^\rho A^{(\rho)}S_{t/2}^\rho u) \circ \Phi^{-1}\|_{M,\gamma,x, E^-,\frac{1}{2}}
} \hspace*{30mm} \\*
& \leq & t^{(2-\delta)/4} \, c \, (t/2)^{-\gamma/4} \, e^{\omega (1+\rho^2)t/2} \, 
   \|A^{(\rho)} \, S_{t/2}^\rho u\|_{L_2(\Omega)}  \\
& \leq & c_0 \, (t/2)^{-1} \, e^{\omega_0 (1+\rho^2) t/2} \, t^{(2-\delta)/4} \, 
   c \, (t/2)^{-\gamma/4} \, e^{\omega (1+\rho^2) t/2} \, \|u\|_{L_2(\Omega)}  \\
& \leq & 2^{1 + \gamma/4} c_0 \, c \, t^{-(\gamma + \delta) / 4}
   \, t^{-1/2} \, e^{(\omega_0 + \omega)(1+\rho^2) t} \, \|u\|_{L_2(\Omega)}
.
\end{eqnarray*}
Secondly, 
\begin{eqnarray*}
\varepsilon^{2-\delta} \, 
   \|(a_0^{(\rho)} \, S_t^\rho u) \circ \Phi^{-1}\|_{M,\gamma,x, E^-,\frac{1}{2}}
& \leq & t^{(2-\delta)/4} \, 4 M \, (1+\rho^2) \, 
   c \, t^{-\gamma / 4} \, e^{\omega (1+\rho^2) t} \, \|u\|_{L_2(\Omega)}  \\
& \leq & 4 c \, M \, t^{-(\gamma + \delta) / 4} \, t^{-1/2} \, 
   e^{(\omega + 1)(1+\rho^2) t} \, \|u\|_{L_2(\Omega)}
.
\end{eqnarray*}
Thirdly,
\begin{eqnarray*}
\lefteqn{
\varepsilon^{2-\delta} \, 
\|(\sum_{i=1}^d a_i^{(\rho)} \, \partial_i S_t u) \circ \Phi^{-1}\|_{M,\gamma,x, E^-,\frac{1}{2}}
} \hspace*{30mm} \\*
& \leq & t^{(2-\delta)/4} \, M (1+|\rho|) \, 
   \|(\nabla S_t^\rho u) \circ \Phi^{-1}\|_{M,\gamma,x, E^-,\frac{1}{2}}  \\
& \leq & 2 t^{(2-\delta)/4} \, M \, t^{-1/2} \, e^{ (1+\rho^2)t} \, K \, 
   \|\nabla ((S_t^\rho u) \circ \Phi^{-1})\|_{M,\gamma,x, E^-,\frac{1}{2}}  \\
& \leq & 2 c \, K \, M \, t^{-(\gamma + \delta) / 4} \, t^{-1/2} \, 
   e^{(\omega + 1)(1+\rho^2) t} \, \|u\|_{L_2(\Omega)}
.
\end{eqnarray*}
The terms with $f_i$ in Proposition~\ref{probin301} can be estimated by
\begin{eqnarray*}
\sum_{i=1}^d \|(b_i^{(\rho)} \, S_t u) \circ \Phi^{-1}\|_{M,\gamma+\delta,x, E^-,\frac{1}{2}}
& \leq & d \, M \, (1+|\rho|) \, 
     \|(S_t^{\rho } u) \circ \Phi^{-1}\|_{M,\gamma+\delta,x, E^-,\frac{1}{2}}  \\
& \leq & 2 d \, M \, c'' \, t^{-(\gamma + \delta) / 4} \, t^{-1/2} \, e^{(1+\rho^2) t} \, 
   e^{\omega'' (1+\rho^2) t} \, \|u\|_{L_2(\Omega)}
,   
\end{eqnarray*}
where we used (\ref{elrobin304;2}) in the last step.
Next, 
\begin{eqnarray*}
\varepsilon^{-(\gamma + \delta)} \, \|\nabla S_t^\rho u\|_{L_2(\Omega)}
& \leq & t^{-(\gamma + \delta) / 4} \, e^{(\gamma + \delta) t} \, c_0 \, t^{-1/2} \, 
   e^{\omega_0 (1+\rho^2) t} \, \|u\|_{L_2(\Omega)}  \\
& \leq & c_0 \, t^{-(\gamma + \delta) / 4} \, t^{-1/2} \, 
   e^{(\omega_0 + d + 2)(1+\rho^2) t} \, \|u\|_{L_2(\Omega)}
.
\end{eqnarray*}
Finally for the new terms,
\begin{eqnarray*}
\varepsilon^{2-\delta} \, \|\nabla ((S_t^\rho u) \circ \Phi^{-1}) \|_{M, \gamma, x, E^-, \frac{1}{2}} 
& \leq & t^{(2-\delta)/4} \, c \, t^{-\gamma /4} 
     \, t^{-1/2} \, e^{\omega (1+\rho^2)t} \, \|u\|_{L_2(\Omega)} \\
& \leq & c \, t^{(\gamma + \delta)/4} \, t^{-1/2} \, e^{(\omega + 1) (1+\rho^2) t} \, \|u\|_{L_2(\Omega)} 
\end{eqnarray*}
and
\begin{eqnarray*}
\|(S_t^\rho u) \circ \Phi^{-1}\|_{M, \tilde{\gamma}, x, E^{-}, \frac{1}{2}} 
& \leq & c'' \, t^{-(\gamma + \delta)/4} \, e^{\omega'' (1+\rho^2) t} \, \|u\|_{L_2(\Omega)} \\
& \leq & c'' \, t^{-(\gamma + \delta)/4} \, t^{-1/2} \, e^{(\omega'' + 1)(1+\rho^2) t} \, \|u\|_{L_2(\Omega)},
\end{eqnarray*}
where (\ref{eprobin303;4}) and (\ref{elrobin304;2}) are used.
Now (\ref{eprobin303;4}) for $P(\gamma + \delta)$ follows from Proposition \ref{probin301}.
\end{proof}

\noindent
{\bf End of proof of Proposition~\ref{probin303}.}
This follows as at the end of the proof of Proposition~7.2 in \cite{ERe2}.
\end{proof}

The part of $\Omega$ away from $\Gamma$ can be estimated by interior regularity.

\begin{prop} \label{probin305}
Let $\Omega \subset \Ri^d$ be open bounded with Lipschitz boundary $\Gamma$.
Let $\zeta >0$.
Then for all $\mu, M>0$ there exist $\kappa \in (0,1)$ and $c,\omega >0$ such that the following is valid.
Let $(C,a,b,a_0) \in \ce (\Omega, \mu, M)$ and $\beta \in L_{\infty}(\Gamma)$ with 
$\|\beta\|_{L_{\infty}(\Gamma)} \leq M$.
Then
\[
\|S_t^\rho u\|_{L_\infty(\Omega_{\zeta})}
\leq c \, t^{-d/4} \, e^{\omega (1+\rho^2) t} \, \|u\|_{L_2(\Omega)}
\]
and 
\[
|(S_t^\rho u)(x) - (S_t^\rho u)(y)|
\leq c \, t^{-d/4} \, t^{-\kappa/2} \, e^{\omega (1+\rho^2) t} \, \|u\|_{L_2(\Omega)} \, |x-y|^\kappa
\]
for all $t>0$, $u\in L_2(\Omega)$, $\rho \in \Ri $, $\psi \in \cd $ and $x,y\in \Omega_\zeta$ with 
$|x-y| \leq 1$, where $\Omega_\zeta = \{ x\in \Omega : d(x,\Gamma) > \zeta \} $.
\end{prop}
\begin{proof}
This follows similarly to the proof of Proposition~\ref{probin303} using 
\cite{ERe2} Proposition~3.2 instead of Proposition~\ref{probin301}.
\end{proof}

We can now prove Gaussian H\"older kernel bounds for second-order 
operators with complex lower-order coefficients and complex Robin boundary conditions.

\begin{thm} \label{trobin306}
Let $\Omega \subset \Ri^d$ be open bounded with Lipschitz boundary $\Gamma$.
Then for all $\mu, M, \tau>0$ and $\tau' \in (0,1)$ there exist $\kappa \in (0,1)$ and 
$b,c, \omega >0$ such that the following is valid.
Let $(C,a,b,a_0)\in \ce (\Omega, \mu, M)$ and $\beta \in L_\infty(\Gamma)$
with $\|\beta\|_{L_\infty(\Gamma)} \leq M$.
Let $S$ be the semigroup generated by $-A$.
Then $S$ has a kernel $K$.
Moreover,
\[|K_t(x,y)| \leq c \, t^{-d/2} \, e^{-b\frac{|x-y|^2}{t}} \, e^{\omega t}\]
and
\[|K_t(x,y) - K_t(x',y')| 
\leq c \, t^{-d/2} \, \Big( \frac{|x-x'|+|y-y'|}{t^{1/2} + |x-y|} \Big)^{\kappa} \,
   e^{-b\frac{|x-y|^2}{t}} \, e^{\omega t} \]
for all $x,x',y,y' \in \Omega $ and $t>0$ with $|x-x'|+|y-y'| \leq \tau \, t^{1/2} + \tau' \, |x-y|$.
\end{thm}
\begin{proof}
By a compactness argument it follows from Propositions~\ref{probin303} and \ref{probin305} 
that there exist $\delta, \kappa \in (0,1)$ and $c,\omega >0$ such that
\[
|(S^\rho_t u)(x)|
\leq c \, t^{-d/4} \, e^{\omega (1 + \rho^2) t} \, \|u\|_{L_2(\Omega)}
\]
and 
\[
|(S^\rho_t u)(x) - (S^\rho_t u)(y)|
\leq c \, t^{-d/4} \, t^{-\kappa/2} \, e^{\omega (1 + \rho^2) t} \, \|u\|_{L_2(\Omega)} \, |x-y|^\kappa
\]
for all $u \in L_2(\Omega)$, $t > 0$, $\rho \in \Ri$ and $\psi \in \cd$ with $|x-y| < \delta$.
Then the H\"older Gaussian kernel bounds follow as in the proof of Lemma A.1 in \cite{EO2}.
\end{proof}

\begin{cor} \label{crobin307}
For all $t > 0$ let $T_t \colon C(\overline \Omega) \to C(\overline \Omega)$ be the 
restriction of $S_t$ to $C(\overline \Omega)$.
Then $T$ is a holomorphic $C_0$-semigroup.
\end{cor}
\begin{proof}
This follows as in the proof of Theorem~4.3 in \cite{Nit4},
with the use of Theorem~\ref{trobin306}.
(Note that there is a gap in the proof of Theorem~4.3 in \cite{Nit4}
in case the condition $\beta \geq 0$ is not valid.)
\end{proof}

Finally note that the semigroup $S$ is irreducible in the following sense.

\begin{prop} \label{probin308}
Suppose that $\Omega$ is connected.
Let $\Omega_1 \subset \Omega$ be measurable.
Suppose that $S_t L_2(\Omega_1) \subset L_2(\Omega_1)$.
Then $|\Omega_1| = 0$ or $|\Omega \setminus \Omega_1| = 0$.
\end{prop}
\begin{proof}
It follows from \cite{Ouh5} Theorem~2.2 that $\one_{\Omega_1} u \in W^{1,2}(\Omega)$
for all $u \in W^{1,2}(\Omega)$.
Then the proposition follows by the discussion on page~106 in \cite{Ouh5}.
\end{proof}

Also the semigroup on $C(\overline \Omega)$ is irreducible.

\begin{prop} \label{probin309}
Suppose that $\Omega$ is connected.
Let $T$ be the $C_0$-semigroup on $C(\overline \Omega)$ as in 
Corollary~\ref{crobin307}.
Let $F \subset \overline \Omega$ be closed and suppose that 
$T_t I \subset I$ for all $t > 0$, where 
$I = \{ u \in C(\overline \Omega) : u|_F = 0 \} $.
Then $F = \emptyset$ or $F = \overline \Omega$.
\end{prop}
\begin{proof}
Suppose that $F \neq \emptyset$ and $F \neq \overline \Omega$.
Define $f \in C(\overline \Omega)$ by $f(x) = d(x,F)$.
Let $t > 0$ and $x \in F$.
If $\tau \in C(\overline \Omega)$, then $f \, \tau \in I$, so
$0 = (T_t (f \, \tau))(x) = \int_\Omega K_t(x,y) \, f(y) \, \tau(y) \, dy$.
Hence $K_t(x,y) \, f(y) = 0$ for almost every $y \in \Omega$
and by continuity for all $y \in \Omega$.
Therefore $K_t(x,y) = 0$ for all $y \in \overline \Omega \setminus F$
and by continuity for all 
$y \in \overline{ \overline \Omega \setminus F}$, where the closure is in $\Ri^d$.
Let $F^\circ$ denote the interior of $F$ in $\Ri^d$.
It is elementary to show that 
$\Omega \setminus (F^\circ) \subset \overline{ \overline \Omega \setminus F}$.
Hence we proved that $K_t(x,y) = 0$ for all $x \in F$, $y \in \Omega \setminus (F^\circ)$
and $t > 0$.

Let $J = \{ u \in L_2(\Omega) : u|_F = 0 \mbox{ a.e.} \} $.
If $u \in J$, $t > 0$ and $x \in F$, then 
\[
(S_t u)(x)
= \int_\Omega K_t(x,y) \, u(y) \, dy
= \int_{\Omega \setminus F} K_t(x,y) \, u(y) \, dy
= 0
.  \]
So $S_t J \subset J$ for all $t > 0$.
Since $S$ is irreducible by Proposition~\ref{probin308}, it follows 
that $|F| = 0$ or $|\overline \Omega \setminus F| = 0$.
Since $F \neq \overline \Omega$ there exists an $x \in \overline \Omega$ and 
$r > 0$ such that $B(x,r) \subset \Ri^d \setminus F$.
Then 
$0 < |\Omega(x,r)| 
\leq |\overline \Omega \setminus F|$.
Hence $|F| = 0$.
Then also $|F^\circ| = 0$ and consequently $F^\circ = \emptyset$.
Therefore $\Omega \setminus (F^\circ) = \Omega$.
It follows that $K_t(x,y) = 0$ for all $t > 0$, $x \in F$ and $y \in \Omega$,
and then by continuity for all $y \in \overline \Omega$.
Then $1 = \lim_{t \downarrow 0} (T_t \one_{\overline \Omega})(x) 
= \lim_{t \downarrow 0} \int_\Omega K_t(x,y) \, dy = 0$ for all $x \in F$.
This is a contradiction since $F \neq \emptyset$.
\end{proof}

\section{Lower kernel bounds}    \label{Srobin5}

In this short section we prove the Gaussian lower bounds of Theorem~\ref{trobin103}.
The general outline is standard. 
We first show on-diagonal lower bounds for small time. 
Secondly we use the H\"older Gaussian upper bounds to obtain 
lower bounds close to the diagonal for small time. 
Finally we use the semigroup property together with the chain condition
to prove Gaussian lower bounds.

Adopt the notation and assumption of Theorem~\ref{trobin103}.
Let $T$ be the $C_0$-semigroup in $C(\overline \Omega)$ as in 
Corollary~\ref{crobin307}.
Then
$\lim_{t \downarrow 0} 
   \|T_t \one_{\overline \Omega} - \one_{\overline \Omega}\|_{C(\overline \Omega)} 
= 0$.
Hence 
\[
\lim_{t \downarrow 0} \sup_{x \in \Omega} 
   \Big| 1 - \int_\Omega K_t(x,y) \, dy \Big| = 0
.  \]
It follows from \cite{ER22} Theorem~2.1 that there are $c_1,c_2,t_0 > 0$
such that 
\[
K_t(x,y)
\geq c_1 \, t^{-d/2}
\]
for all $x,y \in \Omega$ and $t \in (0,t_0]$ with $|x-y| \leq c_2 \, t^{1/2}$.
Without loss of generality we may assume that $t_0 \leq 1$.
By Proposition~\ref{probinapp1} in Appendix~\ref{SrobinA}
the set $\Omega$ satisfies the chain condition.
That is, there exists a $c_3 > 0$ such that for all $x,y \in \Omega$ and 
$n \in \Ni$ there exist $x_0,\ldots,x_n \in \Omega$ such that 
$x_0 = x$, $x_n = y$ and 
$|x_{k+1} - x_k| \leq c_3 \frac{|x-y|}{n}$ for all $k \in \{ 0,\ldots,n-1 \} $.
Since $\Omega$ is bounded and Lipschitz, there exists a $c_4 > 0$ 
such that $|\Omega(x,r)| \geq c_4 \, r^d$ for all $x \in \Omega$ and 
$r \in (0,1]$.

Let $x,y \in \Omega$ and $t > 0$.
Let $n \in \Ni$ be the smallest natural number such that 
\[
\frac{4 c_3^2 \, |x-y|^2}{c_2^2 \, t} \leq n
\quad \mbox{and} \quad
\frac{t}{t_0} \leq n
.  \]
Then 
\begin{equation}
n-1 \leq \frac{4 c_3^2 \, |x-y|^2}{c_2^2 \, t} + \frac{t}{t_0}
.  
\label{etrobin103;1}
\end{equation}
By the chain condition there exist $x_0,\ldots,x_n \in \Omega$ such that 
$x_0 = x$, $x_n = y$ and 
$|x_{k+1} - x_k| \leq \frac{c_3}{n} \, |x-y|$ for all $k \in \{ 0,\ldots,n-1 \} $.
Then the semigroup property gives
\begin{eqnarray*}
K_t(x,y)
& = & \int_\Omega \ldots \int_\Omega K_{\frac{t}{n}}(x,z_1) \, 
    K_{\frac{t}{n}}(z_1,z_2) \ldots K_{\frac{t}{n}}(z_{n-2},z_{n-1}) \, 
    K_{\frac{t}{n}}(z_{n-1},y) \, dz_1 \ldots dz_{n-1} \\
& \geq & \int_{B(x_1, \frac{c_2 \sqrt{t}}{4 \sqrt{n}})}
        \ldots \int_{B(x_{n-1}, \frac{c_2 \sqrt{t}}{4 \sqrt{n}})}
    K_{\frac{t}{n}}(x,z_1) \, 
    K_{\frac{t}{n}}(z_1,z_2) \ldots K_{\frac{t}{n}}(z_{n-2},z_{n-1}) \, 
    K_{\frac{t}{n}}(z_{n-1},y) \\*
& & \hspace*{117mm} dz_1 \ldots dz_{n-1}
.  
\end{eqnarray*}
If $z_k \in B(x_k, \frac{c_2 \sqrt{t}}{4 \sqrt{n}})$ for all 
$k \in \{ 1,\ldots,n-1 \} $ and we set $z_0 = x_0$ and $z_n = x_n$, then 
\[
|z_k - z_{k+1}|
\leq |x_k - x_{k+1}| + \frac{2 c_2 \sqrt{t}}{4 \sqrt{n}}
\leq \frac{c_3}{n} \, |x-y| + \frac{c_2 \sqrt{t}}{2 \sqrt{n}}
\leq \frac{c_3}{n} \, \frac{ c_2 \sqrt{n} \sqrt{t} }{ 2 c_3 }
 + \frac{c_2 \sqrt{t}}{2 \sqrt{n}}
= c_2 \Big( \frac{t}{n} \Big)^{1/2}
\]
for all $k \in \{ 0,\ldots,n-1 \} $
and $\frac{t}{n} \leq t_0$.
Hence $K_{\frac{t}{n}}(z_k,z_{k+1}) \geq c_1 \, n^{d/2} \, t^{-d/2}$
and 
\begin{eqnarray*}
K_t(x,y)
& \geq & \bigg( c_4 \, \Big( \frac{c_2 \sqrt{t}}{4 \sqrt{n}} \Big)^d \bigg)^{n-1}
   \, \Big( c_1 \, n^{d/2} \, t^{-d/2} \Big)^n  \\
& = & c_1 \, (c_1 \, c_2^d \, c_4)^{n-1} \, n^{d/2} \, t^{-d/2}
\geq c_1 \, (c_1 \, c_2^d \, c_4)^{n-1} \, t^{-d/2}
.  
\end{eqnarray*}
Let $M \in [1,\infty)$ be such that $\frac{1}{M} \leq c_1 \, c_2^d \, c_4$.
Then 
\[
(c_1 \, c_2^d \, c_4)^{n-1} 
\geq \Big( \frac{1}{M} \Big)^{n-1}
= e^{-(n-1) \log M}
\geq e^{- (\log M) \, 
   \Big( \textstyle \frac{4 c_3^2 \, |x-y|^2}{c_2^2 \, t} + \frac{t}{t_0} \Big) }
,  \]
where we used (\ref{etrobin103;1}).
Then Theorem~\ref{trobin103} follows.\qed

\section{Dirichlet-to-Neumann operator} \label{Srobin4}

In this section we prove uniform estimates and H\"older continuity 
estimates for the kernel
of the semigroup generated by minus the Dirichlet-to-Neumann operator
on a domain with Lipschitz boundary.
The Dirichlet-to-Neumann operator can be associated to a general second-order 
elliptic differential operator in divergence form with real principal
coefficients and complex lower-order coefficients.
Combining these estimates with Poisson kernel bounds
for operators on $C^{1+\kappa}$-domains and H\"older continuous 
principal coefficients we can prove an extension of Theorem~\ref{trobin102}.

We first introduce the Dirichlet-to-Neumann operator which is 
associated with a general second-order elliptic operator.

Let $\Omega \subset \Ri^d$ be a bounded Lipschitz domain.
For all $k,l \in \{ 1,\ldots,d \} $ let 
$c_{kl} \colon \Omega \to \Ri$ and $a_k,b_k,a_0 \colon \Omega \to \Ci$ 
be bounded and measurable.
Suppose there exists a $\mu > 0$ such that 
$\RRe \sum_{k,l=1}^d c_{kl}(x) \, \xi_k \, \overline{\xi_l} \geq \mu \, |\xi|^2$
for all $x \in \Omega$ and $\xi \in \Ci^d$.
As in Section~\ref{Srobin2}
define the form $\gota \colon W^{1,2}(\Omega) \times W^{1,2}(\Omega) \to \Ci$ by 
\[
\gota(u,v)
= \int_\Omega \sum_{k,l=1}^d c_{kl} (\partial_k u) \, \overline{\partial_l v} 
   + \int_\Omega \sum_{k=1}^d a_k \, (\partial_k u) \, \overline v
   + \int_\Omega \sum_{l = 1}^d b_l \, u \, \overline{\partial_l \, v} 
   + \int_\Omega a_0 \, u \, \overline v
.  \]
Let $A_D$ be the m-sectorial operator in $L_2(\Omega)$ associated with 
$\gota|_{W^{1,2}_0(\Omega) \times W^{1,2}_0(\Omega)}$.
Then $A_D$ is an elliptic operator with Dirichlet boundary conditions.
Throughout this section we assume that $0 \not\in \sigma(A_D)$.

Under the above assumptions one can solve the Dirichlet problem.

\begin{prop} \label{pRobin410}
Let $\varphi \in \Tr W^{1,2}(\Omega)$.
Then there exists a unique $u \in W^{1,2}(\Omega)$ such that $\ca u = 0$
and $\tr u = \varphi$.
\end{prop}
\begin{proof}
See \cite{AE9} Lemma~2.1 (or \cite{BeE3} Lemma~3.2(a)).
\end{proof}

We are now able to define the Dirichlet-to-Neumann operator~$\cn$.
Let $\varphi,\psi \in L_2(\Gamma)$.
Then we say that $\varphi \in \dom \cn$ and $\cn \varphi = \psi$ if there
exists a $u \in W^{1,2}(\Omega)$ such that $\Tr u = \varphi$, $\ca u = 0$
and $\partial_\nu u = \psi$.
The operator $\cn$ can be characterised by the form $\gota$.

\begin{prop} \label{pRobin411}
Let $\varphi,\psi \in L_2(\Gamma)$.
Then the following are equivalent.
\begin{tabeleq}
\item \label{pRobin411-1}
$\varphi \in \dom \cn$ and $\cn \varphi = \psi$.
\item \label{pRobin411-2}
There exists a $u \in W^{1,2}(\Omega)$ such that 
$\Tr u = \varphi$ and 
$\gota(u,v) = (\psi,\Tr v)_{L_2(\Gamma)}$
for all $v \in W^{1,2}(\Omega)$.
\end{tabeleq}
\end{prop}

The proof is left to the reader.

\smallskip

If the form $\gota$ is symmetric, then the operator $\cn$ is self-adjoint
by \cite{AEKS} Theorem~4.5.
The non-symmetric extension is as follows.

\begin{prop} \label{pRobin412}
The operator $\cn$ is m-sectorial.
\end{prop}
\begin{proof}
There exist $\mu_1,\omega_1 > 0$ such that 
\[
\RRe \gota(u)
\geq 2 \mu_1 \, \|u\|_{W^{1,2}(\Omega)}^2 - \omega_1 \, \|u\|_{L_2(\Omega)}^2
\]
for all $u \in W^{1,2}(\Omega)$.
By Proposition~\ref{pRobin410} we can define the map 
$\gamma_D \colon H^{1/2}(\Gamma) \to W^{1,2}(\Omega)$
by $\gamma_D(\varphi) = u$, where $u \in W^{1,2}(\Omega)$ is such that 
$\ca u = 0$ and $\Tr u = \varphi$.
As in \cite{AE2} Section~2 define
$V(\gota ) = \{ u \in W^{1,2}(\Omega) : \gota (u,v) = 0 \mbox{ for all } v \in W^{1,2}_0(\Omega) \} $.
Then $V(\gota)$ is closed in $W^{1,2}(\Omega)$.
If $u \in W^{1,2}(\Omega)$, then
$u = \gamma_D(\Tr u) + (u - \gamma_D(\Tr u)) \in V(\gota) + W^{1,2}_0(\Omega)$.
Therefore $W^{1,2}(\Omega) = V(\gota) + W^{1,2}_0(\Omega)$.
Also $V(\gota) \cap W^{1,2}_0(\Omega) = \{ 0 \} $ since $0 \in \rho(A_D)$.
So $\Tr|_{V(\gota)} \colon V(\gota) \to L_2(\Gamma)$ is injective.
By Ehrling's lemma there exists a $c > 0$ such that 
\[
\|u\|_{L_2(\Omega)}^2
\leq \frac{\mu_1}{\omega_1} \|u\|_{W^{1,2}(\Omega)}^2 + c \, \|\Tr u\|_{L_2(\Omega)}^2
\]
for all $u \in V(\gota)$.
Then 
\begin{equation}
\RRe \gota(u) 
\geq \mu_1 \, \|u\|_{W^{1,2}(\Omega)}^2 - c \, \omega_1 \, \|\Tr u\|_{L_2(\Gamma)}^2
\label{epRobin412;1}
\end{equation}
for all $u \in V(\gota)$.
Now the statement follows from \cite{AE2} Corollary~2.2
and Proposition~\ref{pRobin411}.
\end{proof}

\begin{remark} \label{rRobin413}
Propositions~\ref{pRobin410}, \ref{pRobin411} and \ref{pRobin412}
remain valid if the principle coefficients $c_{kl}$ are 
complex valued.
The proofs are word-by-word the same.
\end{remark}

Let $T$ be the semigroup generated by the operator~$-\cn$.
Recall that we assume that $0 \not\in \sigma(A_D)$.
The main result of this section is the following theorem.

\begin{thm} \label{trobin401}
Suppose $d \geq 3$.
Then there exist $\kappa \in (0,1)$ and $c,\omega > 0$ such that 
$T_t \, L_2(\Gamma) \subset C^\kappa(\Gamma)$,
\begin{equation}
\|T_t\|_{L_2(\Gamma) \to C^\kappa(\Gamma)}
\leq c \, t^{-\frac{d-1}{2}} \, t^{-\kappa} \, e^{\omega t}
\label{etrobin401;1} 
\end{equation}
and 
\begin{equation}
\|T_t\|_{L_2(\Gamma) \to L_\infty(\Gamma)}
\leq c \, t^{-\frac{d-1}{2}} \, e^{\omega t}
\label{etrobin401;2} 
\end{equation}
for all $t > 0$.
\end{thm}

The bounds (\ref{etrobin401;2}) easily follows by interpolation
of the bounds (\ref{etrobin401;1}) and the bounds 
$\|T_t\|_{L_2(\Gamma) \to L_2(\Gamma)} \leq c' \, e^{\omega' t}$.
So it remains to prove the H\"older bounds (\ref{etrobin401;1}).

In case $d=2$, then we can also prove H\"older bounds, but unfortunately 
the singularity in $t$ in (\ref{etrobin401;1}) is not optimal, 
as we loose an~$\varepsilon$.
Since the proof is almost the same, we consider the case $d \geq 2$ in 
the remainder of this section.

If $t \in (0,\infty)$ and $\varphi \in L_2(\Gamma)$, then 
$T_t \varphi \in \dom(\cn)$.
Hence there exists a unique $u_{t,\varphi} \in W^{1,2}(\Omega)$ such that 
$\Tr u_{t,\varphi} = T_t \varphi$ and 
\[
\gota (u_{t,\varphi}, v) = (\cn \, T_t \varphi, \Tr v)_{L_2(\Gamma)}
\]
for all $v \in W^{1,2}(\Omega)$.
The key idea for the proof of (\ref{etrobin401;1}) is to estimate $u_{t,\varphi}$.

\begin{lemma} \label{lrobin402}
There exist $\tilde c_0,\tilde \omega_0 > 0$ such that 
\[
\|u_{t,\varphi}\|_{L_2(\Omega)}
\leq \tilde c_0 \, t^{-1/2} \, e^{\tilde \omega_0 t} \, \|\varphi\|_{L_2(\Gamma)}
\quad \mbox{and} \quad
\|\nabla u_{t,\varphi}\|_{L_2(\Omega)}
\leq \tilde c_0 \,  t^{-1/2}e^{\tilde \omega_0 t} \, \|\varphi\|_{L_2(\Gamma)}
\]
for all $t > 0$ and $\varphi \in L_2(\Gamma)$.
\end{lemma}
\begin{proof}
As in \cite{AE2} Section~2 define
$V(\gota ) = \{ u \in W^{1,2}(\Omega) : \gota (u,v) = 0 \mbox{ for all } v \in W^{1,2}_0(\Omega) \} $.
Let $c, \mu_1,\omega_1 > 0$ be as in (\ref{epRobin412;1}).
Then
\[
\mu_1 \, \|u\|_{W^{1,2}(\Omega)}^2
\leq \RRe \gota (u) + c \, \omega_1 \, \|\Tr u\|_{L_2(\Gamma)}^2
\]
for all $u \in V(\gota)$.
In particular
\[
\mu_1 \, \|u_{t,\varphi}\|_{W^{1,2}(\Omega)}^2
\leq \RRe (\cn \, T_t \varphi, T_t \varphi)_{L_2(\Gamma)}
+ c \, \omega_1 \, \|T_t \varphi\|_{L_2(\Gamma)}^2
\]
and the lemma follows from the analyticity of the semigroup $T$.
\end{proof}

By a compactness argument Theorem~\ref{trobin401} is a consequence of the 
next proposition.

\begin{prop} \label{probin403}
Let $U \subset \Ri^d$ be an open set and $\Phi$ a 
bi-Lipschitz map from an open neighbourhood of $\overline U$ onto 
an open subset of $\Ri^d$ such that $\Phi(U) = E$ and 
$\Phi(\Omega \cap U) = E^-$.
\begin{tabel}
\item \label{pdtnc303-1}
If $d \geq 3$,
then there exist $c,\delta_0,\omega > 0$ and $\kappa \in (0,1)$ such that 
\[
|(T_t \varphi)(x) - (T_t \varphi)(y)|
\leq c \, t^{-\frac{d-1}{2}} \, t^{- \kappa} \, e^{\omega t} \, 
   \|\varphi\|_{L_2(\Gamma)} \, |x-y|^\kappa
\]
for all $t > 0$, $\varphi \in L_2(\Gamma)$ and $x,y \in \Gamma \cap \Phi^{-1}(\frac{1}{2} E)$ 
with $|x-y| \leq \delta_0$.
\item \label{pdtnc303-2}
If $d =2$,
then for all $\varepsilon > 0$ there exist $c,\delta_0,\omega > 0$ and $\kappa \in (0,1)$ such that 
\[
|(T_t \varphi)(x) - (T_t \varphi)(y)|
\leq c \, t^{-\frac{d-1}{2}} \, t^{- \kappa} \, t^{-\varepsilon} \, e^{\omega t} \, 
      \|\varphi\|_{L_2(\Gamma)} \, |x-y|^\kappa
\]
for all $t > 0$, $\varphi \in L_2(\Gamma)$ and $x,y \in \Gamma \cap \Phi^{-1}(\frac{1}{2} E)$ 
with $|x-y| \leq \delta_0$.
\end{tabel}
\end{prop}
\begin{proof}
There exists an $M > 0$ such that 
$(C,a,b,a_0) \in \ce(\Omega,\mu,M)$.
Let $\kappa \in (0,1)$ be as in Proposition~\ref{probin301}.
Let $K \in [1,\infty)$ be larger than the Lipschitz constant 
of $\Phi|_{\Omega \cap U}$ and $\Phi^{-1}|_{E^-}$.
For all $\gamma \in [0,d-2+2\kappa)$ let $P(\gamma)$ be the hypothesis 
\begin{list}{}{\leftmargin=1.8cm \rightmargin=1.8cm \listparindent=12pt}
\item
There exist $c_\gamma,\omega_\gamma > 0$ such that 
\[
\|\nabla (u_{t,\varphi} \circ \Phi^{-1})\|_{M,\gamma,x, E^-,\frac{1}{2}}
\leq c_\gamma \, t^{-\frac{\gamma + 1}{2}} \, e^{\omega_\gamma t} \, \|\varphi\|_{L_2(\Gamma)}
\]
for all $t > 0$, $\varphi \in L_2(\Gamma)$
and $x \in \frac{1}{2} \, E^-$.
\end{list}
Clearly $P(0)$ is valid by Lemma~\ref{lrobin402}.

We need three lemmas.

\begin{lemma} \label{lrobin404}
There exists a $c' > 0$ such that 
\[
\|u\|_{\cm,\gamma+\delta,x,E^-,\frac{1}{2}}
\leq c' \, \Big( \varepsilon^{2-\delta} \|\nabla u\|_{M,\gamma,x,E^-,\frac{1}{2}}
+ \varepsilon^{-(\gamma + \delta - 2)} \|\nabla u\|_{L_2(E^-)} \Big)
\]
for all $\gamma \in [0,d)$ and $\delta \in [0,2]$,
$\varepsilon \in (0,1]$, $u \in W^{1,2}(E^-)$ and $x \in \frac{1}{2} \, E^-$
with $\gamma + \delta \geq 2$.
\end{lemma}
\begin{proof}
By the Neumann type Poincar\'e inequality of \cite{ERe2} Lemma~6.1(a)
there exists a $c > 0$ such that 
\[
\int_{E^-(x_0,R)} |u - \langle u \rangle_{E^-(x_0,R)}|^2
\leq c \, R^2 \int_{E^-(x_0,R)} |\nabla u|^2
\]
for all $x_0 \in \frac{1}{2} \, E^-$, $R \in (0,\frac{1}{2}]$ and 
$u \in W^{1,2}(E^-)$.

Now we prove the lemma.
If $r \in (0,\frac{1}{2} \, \varepsilon^2]$, then
\begin{eqnarray*}
r^{-(\gamma+\delta)} \int_{E^-(x,r)} |u - \langle u\rangle_{E^-(x,r)}|^2
& \leq & c \, r^{2-\delta} r^{-\gamma} \int_{E^-(x,r)} |\nabla u|^2  
\leq  c \,  \varepsilon^{2(2-\delta)} \|\nabla u\|_{M,\gamma,x,E^-,\frac{1}{2}}^2
\end{eqnarray*}
Alternatively,
\[
\int_{E^-(x,r)} |u - \langle u\rangle_{E^-(x,r)}|^2
\leq c \, r^2 \int_{E^-(x,r)} |\nabla u|^2
\leq 2^{\gamma + \delta - 2} \, \varepsilon^{-2(\gamma + \delta - 2)} \, 
     \|\nabla u\|_{L_2(E^-)}^2 \, r^{\gamma+\delta}
\]
if $r \in [\frac{1}{2} \, \varepsilon^2,\frac{1}{2}]$, from which the lemma follows.
\end{proof}

\begin{lemma} \label{lrobin405}
Adopt the assumptions and notation of Proposition~\ref{probin403}.
\begin{tabel}
\item \label{ldtnc306-1}
If $\gamma \in [0,2] \cap [0,d)$, then there exist $c',\omega' > 0$ such that 
\[
\|u_{t,\varphi} \circ \Phi^{-1}\|_{M,\gamma ,x, E^-,\frac{1}{2}}
\leq c' \, t^{- 1/2} \, e^{\omega' t} \, \|\varphi\|_{L_2(\Gamma)}
\]
for all $t > 0$, $\varphi \in L_2(\Gamma)$ and $x \in \frac{1}{2} \, E^-$.
\item \label{ldtnc306-2}
Let $\gamma \in [0,d)$ and $\delta \in [0,2]$ with $\gamma + \delta < d$.
Suppose that $P(\gamma)$ is valid.
Then there exist $c',\omega' > 0$ such that 
\[
\|u_{t,\varphi} \circ \Phi^{-1}\|_{M,\gamma+\delta,x, E^-,\frac{1}{2}}
\leq c' \, t^{- \frac{1 \vee (\gamma+\delta-1)}{2}} \, e^{\omega' t} \, \|\varphi\|_{L_2(\Gamma)}
\]
for all $t > 0$, $\varphi \in L_2(\Gamma)$ and $x \in \frac{1}{2} \, E^-$.
\end{tabel}
\end{lemma}
\begin{proof}
`\ref{ldtnc306-1}'.
By Lemma~\ref{lrobin402} we may assume that $\gamma > 0$.
By the second part of \cite{ERe2} Lemma~6.2 there exists a $c' > 0$ such that 
\begin{eqnarray*}
\|u_{t,\varphi} \circ \Phi^{-1}\|_{\cm,\gamma,x, E^-,\frac{1}{2}}
& \leq & c' \Big( \|\nabla (u_{t,\varphi} \circ \Phi^{-1})\|_{M,0,x, E^-,\frac{1}{2}}
+ \|u_{t,\varphi} \circ \Phi^{-1}\|_{L_2(E^-)} \Big)  \\
& \leq & 2 c' \, d! \, K^{d+1} \, \tilde c_0 \, t^{-1/2} \, e^{\tilde \omega_0 t} \, \|\varphi\|_{L_2(\Gamma)}
\end{eqnarray*}
for all $t > 0$ and $\varphi \in L_2(\Gamma)$,
where $\tilde c_0,\tilde \omega_0 > 0$ are as in Lemma~\ref{lrobin402}.
Then by \cite{ERe2} Lemma~3.1(a) there exist $c'',c''' > 0$ such that 
\begin{eqnarray*}
\|u_{t,\varphi} \circ \Phi^{-1}\|_{M,\gamma,x, E^-,\frac{1}{2}}
& \leq & c'' \, ( \|u_{t,\varphi} \circ \Phi^{-1}\|_{\cm,\gamma,x, E^-,\frac{1}{2}}
+ \|u_{t,\varphi} \circ \Phi^{-1}\|_{L_2(E^-)} )    \\
& \leq & c''' \, t^{-1/2} \, e^{\tilde \omega_0 t} \, \|\varphi\|_{L_2(\Gamma)}
\end{eqnarray*}
for all $t > 0$, $\varphi \in L_2(\Gamma)$ and $x \in \frac{1}{2} E^-$.

`\ref{ldtnc306-2}'.
By Statement \ref{ldtnc306-1} we may assume that $\gamma + \delta \geq 2$.
Let $c' > 0$ be as in Lemma~\ref{lrobin404}.
Choose $\varepsilon = t^{1/2} \, e^{-t} \in (0,1]$.
Then
\begin{eqnarray*}
\lefteqn{
\|u_{t,\varphi} \circ \Phi^{-1}\|_{\cm,\gamma + \delta,x, E^-,\frac{1}{2}}
} \hspace*{1mm}  \\*
& \leq & c' \Big( \varepsilon^{2-\delta} \, 
     \|\nabla (u_{t,\varphi} \circ \Phi^{-1})\|_{M,\gamma,x, E^-,\frac{1}{2}}
   + \varepsilon^{-(\gamma + \delta - 2)} \, \|\nabla (u_{t,\varphi} \circ \Phi^{-1})\|_{L_2(E^-)} \Big)  \\
& \leq & c' \Big( t^{\frac{2-\delta}{2}} \, e^{-(2-\delta) t} \, c_\gamma \, 
     t^{- \frac{\gamma + 1}{2}} \, e^{\omega_\gamma t} \, \|\varphi\|_{L_2(\Gamma)}
   + t^{-\frac{\gamma + \delta - 2}{2}} \, e^{(\gamma + \delta) t} \, d! \, K^{d+1} \, 
     \tilde c_0 \, t^{-\frac{1}{2}} e^{\tilde \omega_0 t} \, \|\varphi\|_{L_2(\Gamma)}  \Big)   \\
& \leq & c'' \, t^{- \frac{\gamma+\delta-1}{2}} \, e^{\omega' t} \, \|\varphi\|_{L_2(\Gamma)}
\end{eqnarray*}
for all $t > 0$, $\varphi \in L_2(\Gamma)$ and $x \in \frac{1}{2} E^-$,
with suitable $c'',\omega' > 0$.

Finally, by \cite{ERe2} Lemma~3.1(a) 
there exist $c''',c'''',\omega'' > 0$ such that 
\begin{eqnarray*}
\|u_{t,\varphi} \circ \Phi^{-1}\|_{M,\gamma + \delta,x, E^-,\frac{1}{2}}
& \leq & c''' \, ( \|u_{t,\varphi} \circ \Phi^{-1}\|_{\cm,\gamma + \delta,x, E^-,\frac{1}{2}}
+ \|u_{t,\varphi} \circ \Phi^{-1}\|_{L_2(E^-)} )  \nonumber  \\
& \leq & c'''' \, t^{- \frac{\gamma+\delta-1}{2}} \, e^{\omega'' t} \, \|\varphi\|_{L_2(\Gamma)}
\end{eqnarray*}
for all $t > 0$, $\varphi \in L_2(\Gamma)$ and $x \in \frac{1}{2} E^-$,
and the lemma follows.
\end{proof}

\begin{lemma} \label{lrobin406}
Let $\gamma \in [0,d-2 + 2\kappa)$ and suppose that $P(\gamma)$ is 
valid.
Let $\delta \in (0,2]$ and suppose that $\gamma + \delta < d-2+2\kappa$.
Then one has the following.
\begin{tabel}
\item \label{ldtnc304-1}
If $d \geq 3$, then $P(\gamma+\delta)$ is valid.
\item \label{ldtnc304-2}
If $d = 2$, then for all $\eta > 0$ there exist $c',\omega' > 0$ such that 
\[
\|\nabla (u_{t,\varphi} \circ \Phi^{-1})\|_{M,\gamma + \delta,x, E^-,\frac{1}{2}}
\leq c' \, t^{-\frac{\gamma + \delta + 1}{2}} 
\, t^{- \eta} 
\, e^{\omega' t} \, \|\varphi\|_{L_2(\Gamma)}
\]
for all $t > 0$, $\varphi \in L_2(\Gamma)$ and $x \in \frac{1}{2} \, E^-$.
\end{tabel}
\end{lemma}
\begin{proof}
Without loss of generality we may assume in case $d = 2$ that 
$2 \kappa \leq 2 - 2 \eta$ and hence $\gamma + \delta \leq 2 - 2 \eta$.
Define $\tilde \gamma \in [\gamma+\delta,d)$ by 
\[
\tilde \gamma
= \left\{ \begin{array}{ll}
\gamma + \delta & \mbox{if } \gamma + \delta \geq 2 ,  \\[5pt]
2 & \mbox{if } \gamma + \delta < 2 \mbox{ and } d \geq 3 ,  \\[5pt]
2 - 2 \eta & \mbox{if } d = 2 .
\end{array} \right.
\]
Note that $\tilde \gamma \geq 2$ if $d \geq 3$.
Let $c > 0$ be as in Proposition~\ref{probin301}
with the choice $\beta = \one_\Gamma$.
By analyticity of $T$ there exist $\tilde c,\tilde \omega > 0$ such that 
$\|\cn \,T_t  \varphi\|_{L_2(\Gamma)}
\leq \tilde c \, t^{-1} \, e^{\tilde \omega t} \, \|\varphi\|_{L_2(\Gamma)}$
for all $t > 0$ and $\varphi \in L_2(\Gamma)$.
By Lemma~\ref{lrobin405}\ref{ldtnc306-2} there exist $\hat c, \hat \omega > 0$ such that 
\begin{eqnarray*}
\|u_{t,\varphi} \circ \Phi^{-1}\|_{M,\gamma,x, E^-,\frac{1}{2}}
& \leq & \hat c \, t^{- \frac{1 \vee (\gamma-1)}{2}} \, e^{\hat \omega t} \, 
       \|\varphi\|_{L_2(\Gamma)}, \\
\|u_{t,\varphi} \circ \Phi^{-1}\|_{M,\gamma+\delta,x, E^-,\frac{1}{2}}
& \leq & \hat{c} \, t^{- \frac{1 \vee (\gamma+\delta-1)}{2}} \, e^{\hat \omega t} \, 
       \|\varphi\|_{L_2(\Gamma)}
\mbox{ and}  \\
\|u_{t,\varphi} \circ \Phi^{-1}\|_{M,\tilde \gamma,x, E^-,\frac{1}{2}}
& \leq & \hat c \, t^{- \frac{1 \vee (\tilde \gamma-1)}{2}} \, e^{\hat \omega t} \, 
      \|\varphi\|_{L_2(\Gamma)}
\end{eqnarray*}
for all $t > 0$, $\varphi \in L_2(\Gamma)$ and $x \in \frac{1}{2} \, E^-$.

Let $t > 0$, $\varphi \in L_2(\Gamma)$ and $x \in \frac{1}{2} E^-$.
Since $\cn \, T_{2t} \varphi = T_t \, \cn \, T_t \varphi$ it follows that 
\[
\gota_p(u_{2t,\varphi},v)
= (f, v)_{L_2(\Omega)} + \sum_{k=1}^d (f_k, \partial_k v)_{L_2(\Omega)}
   + \int_\Gamma (\tr u_{t, \cn \, T_t \varphi}) \, \overline{\tr v}
\]
for all $v \in W^{1,2}(\Omega)$, where 
$f = - a_0 \, u_{2t, \varphi} - \sum_{k=1}^d a_k \, \partial_k u_{2t, \varphi}$ and 
$f_k = - b_k \, u_{2t, \varphi}$.
Hence Proposition~\ref{probin301} with the choice $\varepsilon = t^{1/2} \, e^{-t} \in (0,1]$
gives
\begin{eqnarray*}
\lefteqn{
\|\nabla(u_{2t,\varphi} \circ \Phi^{-1})\|_{M,\gamma+\delta, x, E^-, \frac{1}{2}}
} \hspace*{10mm} \\*
& \leq & c \Big( \varepsilon^{2-\delta} \, 
     \|(a_0 \, u_{2t,\varphi}) \circ \Phi^{-1}\|_{M,\gamma, x, E^-, \frac{1}{2}}
   + \varepsilon^{2-\delta} \sum_{k=1}^d 
     \|(a_k \, \partial_k u_{2t,\varphi}) \circ \Phi^{-1}\|_{M,\gamma, x, E^-, \frac{1}{2}} \\*
& & \hspace*{10mm} {} 
+ \sum_{k=1}^d \|( b_k \, u_{2t, \varphi})\circ  \Phi^{-1}\|_{M, \gamma +\delta, x, E^-, \frac{1}{2}}
+ \varepsilon^{-(\gamma + \delta)} \, \|\nabla u_{2t,\varphi}\|_{L_2(\Omega)} \\*
& & \hspace*{10mm} {} 
+ \varepsilon^{2-\delta} \,
   \|\nabla(u_{t,\cn \, T_t \varphi} \circ \Phi^{-1})\|_{M,\gamma, x, E^-, \frac{1}{2}}
+ \varepsilon^{\tilde \gamma - \gamma - \delta} \, 
       \|u_{t,\cn \, T_t \varphi} \circ \Phi^{-1}\|_{M,\tilde \gamma, x, E^-, \frac{1}{2}} 
\Big)
.
\end{eqnarray*}
We estimate the terms.

First 
\begin{eqnarray*}
\varepsilon^{2-\delta} \, \|(a_0 \, u_{2t,\varphi}) \circ \Phi^{-1}\|_{M,\gamma, x, E^-, \frac{1}{2}}
& \leq & M \, t^{\frac{2-\delta}{2}} \, 
     \hat c \, (2t)^{- \frac{1 \vee (\gamma-1)}{2}} \, 
          e^{2 \hat \omega t} \, \|\varphi\|_{L_2(\Gamma)}  \\
& \leq & \hat c \, M \, t^{\frac{2-\delta}{2}} \, t^{- \frac{1 \vee (\gamma-1)}{2}}
     \, e^{2 \hat \omega t} \, \|\varphi\|_{L_2(\Gamma)}  \\
& = & \hat c \, M \, t^{-\frac{\gamma + \delta + 1}{2}} \, 
     t^{\frac{2+ (2 \wedge \gamma)}{2}}
     \, e^{2 \hat \omega t} \, \|\varphi\|_{L_2(\Gamma)}  \\
& \leq & \hat c \, M \, t^{-\frac{\gamma + \delta + 1}{2}} \, 
      e^{(2 \hat \omega + 1+\gamma) t} \, \|\varphi\|_{L_2(\Gamma)} 
.  
\end{eqnarray*}
Secondly, the induction hypothesis $P(\gamma)$ gives
\begin{eqnarray*}
\varepsilon^{2-\delta} \sum_{k=1}^d 
    \|(a_k \, \partial_k u_{2t,\varphi}) \circ \Phi^{-1}\|_{M,\gamma, x, E^-, \frac{1}{2}}
& \leq & t^{\frac{2-\delta}{2}} \, d \, M \, 
      \|(\nabla u_{2t, \varphi}) \circ \Phi^{-1}\|_{M, \gamma, x, E^-, \frac{1}{2}} \\
& \leq & d \, M \, t^{\frac{2-\delta}{2}} \, K \, 
      \|\nabla (u_{2t, \varphi} \circ \Phi^{-1})\|_{M, \gamma, x, E^-, \frac{1}{2}} \\
& \leq & d \, K \, M \, t^{\frac{2-\delta}{2}} \, 
     c_{\gamma} \, (2t)^{-\frac{\gamma +1}{2}} \, 
     e^{2 \omega_{\gamma}t} \, \|\varphi \|_{L_2(\Gamma)} \\
& = & c_{\gamma} \, d \, K \, M \, t^{-\frac{\gamma + \delta + 1}{2}} 
    \, t \, e^{2 \omega_\gamma t} \, \|\varphi \|_{L_2(\Gamma)} \\
& \leq & c_{\gamma} \, d \, K \, M \, t^{-\frac{\gamma + \delta + 1}{2}}
    \, e^{(2 \omega_\gamma +1) t} \,\|\varphi \|_{L_2(\Gamma)}
.
\end{eqnarray*}
Thirdly,
\begin{eqnarray*}
\sum_{k=1}^d \|(b_k \, u_{2t, \varphi})\circ  \Phi^{-1}\|_{M, \gamma +\delta, x, E^-, \frac{1}{2}}
& \leq & d \, M \, \|u_{2t, \varphi}\circ  \Phi^{-1}\|_{M, \gamma +\delta, x, E^-, \frac{1}{2}} \\
& \leq & d \, M \, \hat c \, (2t)^{-\frac{1\vee (\gamma + \delta -1)}{2}} 
    \, e^{2 \hat \omega t} \, \|\varphi \|_{L_2(\Gamma)} \\
& \leq & \hat c \, d \, M \, t^{-\frac{\gamma + \delta + 1}{2}} \, 
      t^{\frac{2\wedge (\gamma + \delta)}{2}} \, e^{2 \hat \omega t} \, \|\varphi \|_{L_2(\Gamma)} \\
& \leq & \hat c \, d \, M \, t^{-\frac{\gamma + \delta + 1}{2}} \, 
     e^{(2 \hat \omega +1 + \gamma + \delta) t} \, \|\varphi \|_{L_2(\Gamma)} \\
.
\end{eqnarray*}
Fourthly, 
\begin{eqnarray*}
\varepsilon^{-(\gamma + \delta)} \, \|\nabla u_{2t,\varphi}\|_{L_2(\Omega)}
& \leq & t^{-\frac{\gamma + \delta}{2}} \, e^{(\gamma + \delta) t} \,
     \tilde c_0 \, (2t)^{-1/2} \, e^{2 \tilde \omega_0 t} \, \|\varphi\|_{L_2(\Gamma)}  \\
& \leq & \tilde c_0 \, t^{-\frac{\gamma + \delta + 1}{2}} \, 
     e^{(2 \tilde \omega_0 + \gamma + \delta) t} \, \|\varphi\|_{L_2(\Gamma)}
,
\end{eqnarray*}
where $\tilde c_0,\tilde \omega_0 > 0$ are as in Lemma~\ref{lrobin402}.
Fifthly, 
\begin{eqnarray*}
\varepsilon^{2-\delta} \, 
     \|\nabla  (u_{t,\cn \, T_t \varphi} \circ \Phi^{-1})\|_{M,\gamma, x, E^-, \frac{1}{2}} 
& \leq & t^{\frac{2-\delta}{2}} \, M \, 
     c_\gamma \, t^{-\frac{\gamma + 1}{2}} \, e^{\omega_\gamma t} \, 
        \|\cn \, T_t \varphi\|_{L_2(\Gamma)}  \\
& \leq & t^{\frac{2-\delta}{2}} \, M \, 
     c_\gamma \, t^{-\frac{\gamma + 1}{2}} \, e^{\omega_\gamma t} \, 
     \tilde c \, t^{-1} \, e^{\tilde \omega t} \, \|\varphi\|_{L_2(\Gamma)}  \\
& = & \tilde c \, c_\gamma \, M \, t^{-\frac{\gamma + \delta + 1}{2}} \, 
     e^{(\omega_\gamma + \tilde \omega) t} \, \|\varphi\|_{L_2(\Gamma)}
.
\end{eqnarray*}
Finally, 
\begin{eqnarray*}
\varepsilon^{\tilde \gamma - \gamma - \delta} \, 
     \|u_{t,\cn \, \, T_t \varphi} \circ \Phi^{-1}\|_{M,\tilde \gamma, x, E^-, \frac{1}{2}}
& \leq & t^{ \frac{ \tilde \gamma - \gamma - \delta }{2} } \, M
     \, \hat c \, t^{- \frac{1 \vee (\tilde \gamma - 1)}{2}} \, e^{\hat \omega t} 
     \, \|\cn \, T_t \varphi\|_{L_2(\Gamma)}  \\
& \leq & t^{ \frac{ \tilde \gamma - \gamma - \delta }{2} } \, M
     \, \hat c \, t^{- \frac{1 \vee (\tilde \gamma - 1)}{2}} \, e^{\hat \omega t} \, 
    \tilde c \, t^{-1} \, e^{\tilde \omega t} \, \|\varphi\|_{L_2(\Gamma)}  \\
& = & \tilde c \, \hat c \, M \, t^{-\frac{\gamma + \delta + 1}{2}} \,
     t^{- \frac{ (2-\tilde \gamma) \vee 0 }{2} } \,
     e^{(\hat \omega + \tilde \omega) t} \, \|\varphi\|_{L_2(\Gamma)}
.
\end{eqnarray*}
Note that $t^{- \frac{ (2-\tilde \gamma) \vee 0 }{2} } = 1$ 
if $d \geq 3$, since then $\tilde \gamma \geq 2$.

The lemma follows.
\end{proof}

Now we are able to complete the proof of Proposition~\ref{probin403}.

\medskip

\noindent
{\bf End of proof of Proposition~\ref{probin403}.}

`\ref{pdtnc303-1}'.
(Suppose that $d \geq 3$.)
We know that $P(0)$ is valid.
Then it follows by induction from Lemma~\ref{lrobin406}\ref{ldtnc304-1} 
that $P(\gamma)$ is valid for all $\gamma \in [0,d-2+2\kappa)$.
In particular $P(d-2+\kappa)$ is valid.
Hence using Lemma~\ref{lrobin404} with $\delta = 2$ and $\varepsilon = t^{1/2} \, e^{-t}$
one deduces that there are $c,\omega > 0$ such that 
\[
\|u_{t,\varphi} \circ \Phi^{-1}\|_{\cm,d+\kappa,x, E^-,\frac{1}{2}}
\leq c \, t^{-\frac{d-1+\kappa}{2}} \, e^{\omega t} \, \|\varphi\|_{L_2(\Gamma)}
\]
for all $t > 0$, $\varphi \in L_2(\Gamma)$ and $x \in \frac{1}{2} \, E^-$.
Therefore the function $(u_{t,\varphi} \circ \Phi^{-1})|_{\frac{1}{2} \, E^-}$ has a 
continuous representative, which is H\"older continuous and it extends 
continuously to the closure of $\frac{1}{2} \, E^-$.
By \cite{ERe2} Lemma~3.1(c) there exists a $c' > 0$ such that 
\[
|(u_{t,\varphi} \circ \Phi^{-1})(x) - (u_{t,\varphi} \circ \Phi^{-1})(y)|
\leq c' \, t^{-\frac{d-1+\kappa}{2}} \, e^{\omega t} \, \|\varphi\|_{L_2(\Gamma)}
     \, |x-y|^{\kappa / 2}
\]
for all $t > 0$, $\varphi \in L_2(\Gamma)$ and $x,y \in \frac{1}{2} \, E^-$
with $|x-y| < \frac{1}{4}$.
The latter estimates extend to all $x,y \in \frac{1}{2} \, \overline{E^-}$
with $|x-y| \leq \frac{1}{4}$.
Since $\Phi$ is bi-Lipschitz with Lipschitz constants $K$, it follows that 
\[
|u_{t,\varphi}(x) - u_{t,\varphi}(y)|
\leq c' \, K^{\kappa / 2} \, t^{-\frac{d-1+\kappa}{2}} \, e^{\omega t} \, \|\varphi\|_{L_2(\Gamma)}
     \, |x-y|^{\kappa / 2}
\]
for all $t > 0$, $u \in L_2(\Gamma)$ and $x,y \in \Gamma \cap \Phi^{-1}(\frac{1}{2} E)$ 
with $|x-y| \leq \frac{1}{4K}$ and Statement~\ref{pdtnc303-1} follows.

The prove of Statement~\ref{pdtnc303-2} is similar by using Lemma~\ref{lrobin406}\ref{ldtnc304-2}.
\end{proof}

The uniform H\"older bounds of Theorem~\ref{trobin401} can be combined with the 
Poisson kernel bounds of \cite{EO6} to obtain H\"older Poisson kernel 
bounds in case the domain $\Omega$ is of class $C^{1 + \kappa}$ for some 
$\kappa > 0$.

\begin{thm} \label{trobin408}
Assume $d \geq 3$.
Suppose there exists a $\kappa > 0$ such that $\Omega$ is of class 
$C^{1 + \kappa}$.
Assume that $c_{kl} = c_{lk}$ is H\"older continuous for all $k,l \in \{ 1,\ldots,d \} $,
the function $a_0$ is real valued and 
$a_k = b_k = 0$ for all $k \in \{ 1,\ldots,d \} $.
Suppose that $0 \not\in \sigma(A_D)$.
Let $K$ be the kernel of the semigroup on $L_2(\Gamma)$ generated by $-\cn$,
where $\cn$ is the Dirichlet-to-Neumann operator.
Then for all $\varepsilon,\tau' \in (0,1)$ and $\tau > 0$ there exist $c,\nu > 0$ 
such that 
\[
|K_t(x,y) - K_t(x',y')|
\leq c \, (t \wedge 1)^{-(d-1)} \, 
   \Big( \frac{|x-x'|+|y-y'|}{t + |x-y|} \Big)^\nu \,
    \frac{1}{\displaystyle \Big( 1 + \frac{|x-y|}{t} \Big)^{d-\varepsilon}} \, 
    (1 + t)^\nu \, e^{- \lambda_1 t}
\]
for all $x,y,x',y' \in \Gamma$ and $t > 0$ with 
$|x-x'| + |y-y'| \leq \tau \, t + \tau' \, |x-y|$,
where $\lambda_1 = \min \sigma(\cn)$.
\end{thm}
\begin{proof}
It follows from Theorem~\ref{trobin401} and 
\cite{EO6} Theorem~1.1 that there exist $c,\omega > 0$ and 
sufficiently small $\nu' \in (0,1)$ 
such that 
\[
|K_t(x,y)|
\leq c \, t^{-(d-1)} \,
    \frac{1}{\displaystyle \Big( 1 + \frac{|x-y|}{t} \Big)^d} \, 
    e^{\omega t}
\]
and 
\begin{equation}
|K_t(x,y) - K_t(x',y)|
\leq c \, t^{-(d-1)} \,
   \Big( \frac{|x-x'|}{t} \Big)^{\nu'} \, e^{\omega t}
\label{etrobin408;1}
\end{equation}
for all $x,y,x' \in \Gamma$ and $t > 0$ with $|x-x'| \leq 1$.
By duality, we obtain similarly, without loss of generality, 
that 
\begin{equation}
|K_t(x,y) - K_t(x,y')|
\leq c \, t^{-(d-1)} \,
   \Big( \frac{|y-y'|}{t} \Big)^{\nu'} \, e^{\omega t}
\label{etrobin408;2}
\end{equation}
for all $x,y,y' \in \Gamma$ and $t > 0$ with $|y-y'| \leq 1$.

Now let $x,y,x',y' \in \Gamma$ and suppose that
$|x-x'| + |y-y'| \leq \tau \, t + \tau' \, |x-y|$.
Then $|x-y| \leq |x' - y'| + \tau \, t + \tau' \, |x-y|$, 
so $|x-y| \leq \frac{1}{1-\tau'} \, |x' - y'| + \frac{\tau}{1-\tau'} \, t$.
Hence
\[
1 + \frac{|x-y|}{t}
\leq 1 + \frac{1}{1-\tau'} \, \frac{|x'-y'|}{t} + \frac{\tau}{1-\tau'}
\leq \frac{1 + \tau}{1-\tau'} \Big( 1 + \frac{|x'-y'|}{t} \Big)
\]
and 
\[
|K_t(x',y')|
\leq c \, t^{-(d-1)} \,
    \frac{1}{\displaystyle \Big( 1 + \frac{|x'-y'|}{t} \Big)^d} \, 
    e^{\omega t}
\leq c \, \frac{(1 + \tau)^d}{(1-\tau')^d} \, t^{-(d-1)} \,
    \frac{1}{\displaystyle \Big( 1 + \frac{|x-y|}{t} \Big)^d} \, 
    e^{\omega t}
.  \]
Therefore
\[
|K_t(x,y) - K_t(x',y')|
\leq 2 c \, \frac{(1 + \tau)^d}{(1-\tau')^d} \, t^{-(d-1)} \,
    \frac{1}{\displaystyle \Big( 1 + \frac{|x-y|}{t} \Big)^d} \, 
    e^{\omega t}
.  \]
Next, it follows from (\ref{etrobin408;1}) and (\ref{etrobin408;2})
that 
\[
|K_t(x,y) - K_t(x',y')|
\leq 2 c \, t^{-(d-1)} \,
   \Big( \frac{|x-x'| + |y-y'|}{t} \Big)^{\nu'} \, e^{\omega t}
.  \]
Then 
\[
|K_t(x,y) - K_t(x',y')|
\leq c' \, t^{-(d-1)} \, 
   \Big( \frac{|x-x'| + |y-y'|}{t} \Big)^{\nu' \varepsilon} 
    \frac{1}{\displaystyle \Big( 1 + \frac{|x-y|}{t} \Big)^{d(1-\varepsilon)}} \, 
   \, e^{\omega t}
\]
by interpolation, where 
$c' = 2 c \, \frac{(1 + \tau)^{d(1-\varepsilon)}}{(1 - \tau')^{d(1-\varepsilon)}}$.
Note that 
\[
\frac{1}{t} 
= \frac{1}{t + |x-y|} \, \Big( 1 + \frac{|x-y|}{t} \Big)
.  \]
Therefore 
\[
|K_t(x,y) - K_t(x',y')|
\leq c' \, t^{-(d-1)} \, 
   \Big( \frac{|x-x'| + |y-y'|}{t + |x-y|} \Big)^{\nu' \varepsilon} 
    \frac{1}{\displaystyle \Big( 1 + \frac{|x-y|}{t} \Big)^{d(1-\varepsilon) - \nu' \varepsilon}} \, 
   \, e^{\omega t}
\]
and the required bounds follow if $t \in (0,3]$.

Finally, there exist $c > 0$ and $\nu \in (0,1)$ such that 
\[
\|T_t\|_{L_1(\Gamma) \to C^\nu(\Gamma)}
\leq \|T_1\|_{L_2(\Gamma) \to C^\nu(\Gamma)} 
      \, \|T_{t-2}\|_{L_2(\Gamma) \to L_2(\Gamma)}
      \, \|T_1\|_{L_1(\Gamma) \to L_2(\Gamma)}
\leq c \, e^{-\lambda_1 t}
\]
for all $t \geq 3$.
Hence 
\[
|K_t(x,y) - K_t(x',y)|
\leq c \, e^{-\lambda_1 t} \, |x-x'|^\nu
\]
for all $x,x',y \in \Gamma$ and $t \geq 3$ with $|x-x'| \leq 1$.
By duality there exists a $c' > 0$ such that 
\[
|K_t(x,y) - K_t(x',y')|
\leq c' \, e^{-\lambda_1 t} \, (|x-x'| + |y-y'|)^\nu
\]
for all $x,x',y,y' \in \Gamma$ and $t \geq 3$ with $|x-x'| \leq 1$
and $|y-y'| \leq 1$.
Since $\Gamma$ is bounded, the required H\"older Poisson bounds
follow for $t \geq 3$.
\end{proof}

\appendix

\section{The chain condition} \label{SrobinA}

Let $\Omega \subset \Ri^d$ be open and connected.
We say that $\Omega$ satisfies the {\bf chain condition} if there exists a
$c > 0$ such that for all $x,y \in \Omega$ and $n \in \Ni$ there are 
$x_0,\ldots,x_n \in \Omega$ such that $x_0 = x$, $x_n = y$ and 
$|x_{k+1} - x_k| \leq \frac{c}{n} \, |x-y|$ for all $k \in \{ 0,\ldots,n-1 \} $.
Obviously in general $\Omega$ does not satisfy the chain condition.

\begin{prop} \label{probinapp1}
Let $\Omega \subset \Ri^d$ be open bounded connected with Lipschitz boundary.
Then $\Omega$ satisfies the chain condition.
\end{prop}

The proof requires some preparation.
Let $\Omega \subset \Ri^d$ be open bounded connected with Lipschitz boundary.
If $T > 0$ and $\gamma \colon [0,T] \to \Omega$ is a Lipschitz curve, 
then $\gamma$ is differentiable almost everywhere.
We define the {\bf length} of $\gamma$ by 
$\ell(\gamma) = \int_0^T |\gamma'(t)| \, dt$.
Define the {\bf geometric distance} $d \colon \Omega \times \Omega \to [0,\infty)$
by $d(x,y)$ is the infimum of $\ell(\gamma)$, where $T > 0$ and 
$\gamma \colon [0,T] \to \Omega$ is a Lipschitz curve with $\gamma(0) = x$ 
and $\gamma(T) = y$.
Obviously $|x-y| \leq \ell(\gamma)$ and hence $|x-y| \leq d(x,y)$.

We first consider a special Lipschitz chart.

\begin{lemma} \label{lrobinapp2}
Let $U \subset \Ri^d$ be an open set and $\Phi$ be a bi-Lipschitz map
from an open neighbourhood of $\overline U$ onto an open 
subset of $\Ri^d$ such that $\Phi(U) = E$ and $\Phi(\Omega \cap U) = E^-$.
Then there are $c_1,c_2 > 0$ such that 
$d(x,y) \leq c_1 \, |x-y|$ and $|x-y| \leq c_2$
for all $x,y \in \Omega \cap U$.
\end{lemma}
\begin{proof}
Let $L \in \Ri$ be larger than both the Lipschitz constant for 
$\Phi$ and $\Phi^{-1}$.
Further, let $x,y \in \Omega \cap U$.
Define $\gamma \colon [0,1] \to \Omega$ by 
$\gamma(t) = \Phi^{-1} ( (1-t) \, \Phi(x) + t \, \Phi(y) )$.
Then $\gamma(0) = x$ and $\gamma(1) = y$.
Moreover, $\gamma$ is Lipschitz continuous and 
$|\gamma'(t)| \leq L \, |\Phi(y) - \Phi(x)| \leq L^2 \, |y-x|$
for almost every $t \in [0,1]$.
So $d(x,y) \leq \ell(\gamma) \leq L^2 \, |x-y|$.
Also $|x-y| \leq L \, |\Phi(x) - \Phi(y)| \leq 2 L$.
\end{proof}

We next show that the geometric distance is equivalent with the induced 
Euclidean distance on $\Omega$.

\begin{lemma} \label{lrobinapp3}
There exists a $c > 0$ such that 
$|x-y| \leq d(x,y) \leq c \, |x-y|$ for all $x,y \in \Omega$.
\end{lemma}
\begin{proof}
By a compactness argument there are $N \in \Ni$ and for all $k \in \{ 1,\ldots,N \} $ 
there are open $U_k \subset \Ri^d$ and a bi-Lipschitz map $\Phi_k$
from an open neighbourhood of $\overline{U_k}$ onto an open 
subset of $\Ri^d$ such that $\Phi_k(U_k) = E$ and $\Phi_k(\Omega \cap U_k) = E^-$;
and moreover, 
$\Gamma \subset \bigcup_{k=1}^N U_k$.
For all $k \in \{ 1,\dots,N \} $ fix $w_k \in \Omega \cap U_k$.
Again by compactness there are $N' \in \{ N+1,N+2,\ldots \} $ and 
for all $k \in \{ N+1,\ldots,N' \} $ there are $w_k \in \Omega$ and 
$r_k > 0$ such that $B(w_k,r_k) \subset \Omega$ and 
\[
\overline \Omega \subset \bigcup_{k=1}^N  U_k \cup \bigcup_{k=N+1}^{N'} B(w_k,r_k)
.  \]
By Lemma~\ref{lrobinapp2} there are $c_1,c_2 \geq 1$ such that 
$d(x,y) \leq c_1 \, |x-y|$ and $|x-y| \leq c_2$ for all $k \in \{ 1,\ldots,N \} $
and $x,y \in \Omega \cap U_k$.
Without loss of generality we may assume that $2 r_k \leq c_2$ for all 
$k \in \{ N+1,\ldots,N' \} $.
For simplicity write $U_k = B(w_k,r_k)$ for all $k \in \{ N+1,\ldots,N' \} $.
Then $d(x,y) \leq c_1 \, |x-y|$ and $|x-y| \leq c_2$ for all 
$k \in \{ N+1,\ldots,N' \} $ and $x,y \in U_k$.

We next prove that the geometric distance $d$ is bounded on $\Omega$.
Define $M = 2 c_2 + \max \{ d(w_k,w_l) : k,l \in \{ 1,\ldots,N' \} \} $.
Let $x,y \in \Omega$.
Then there are $k,l \in \{ 1,\ldots,N' \} $ such that 
$x \in U_k$ and $y \in U_l$.
Hence $d(x,y) \leq d(x,w_k) + d(w_k,w_l) + d(w_l,y) \leq M$.
Therefore $d$ is bounded by~$M$.

Finally suppose that there is no $c > 0$ such that 
$d(x,y) \leq c \, |x-y|$ for all $x,y \in \Omega$.
Then for all $n \in \Ni$ there are $x_n,y_n \in \Omega$ such that 
$d(x_n,y_n) > n \, |x_n - y_n|$.
It follows that $|x_n - y_n| \leq \frac{M}{n}$ for all $n \in \Ni$.
The sequence $(x_n)_{n \in \Ni}$ is bounded since $\Omega$ is bounded. 
Passing to a subsequence if necessary, we may assume that the 
sequence $(x_n)_{n \in \Ni}$ is convergent. 
Let $x = \lim_{n \to \infty} x_n$.
Then $\lim_{n \to \infty} y_n = x$ and $x \in \overline \Omega$.
Since $\overline \Omega \subset \bigcup_{k=1}^N  U_k \cup \bigcup_{k=N+1}^{N'} B(w_k,r_k)$,
there exists a $k \in \{ 1,\ldots,N' \} $ such that $x \in U_k$.
Because $U_k$ is open there exists an $N_0 \in \Ni$ such that 
$x_n \in U_k$ and $y_n \in U_k$ for all $n \in \Ni$ with $n \geq N_0$.
Finally choose $n \in \Ni$ such that $n \geq \max \{ N_0,c_1 \} $.
Then 
\[
n \, |x_n - y_n|
< d(x_n,y_n)
\leq c_1 \, |x_n - y_n|
\leq n \, |x_n - y_n|
.  \]
This is a contradiction.
\end{proof}

Now we are able to prove the proposition.

\begin{proof}[{\bf Proof of Proposition~\ref{probinapp1}.}]
Let $c > 0$ be as in Lemma~\ref{lrobinapp3}.
Let $x,y \in \Omega$ and $n \in \Ni$.
Since the case $x=y$ is trivial, we may assume that $x \neq y$.
There exist $T > 0$ and a Lipschitz curve $\gamma \colon [0,T] \to \Omega$
such that $\gamma(0) = x$, $\gamma(1) = y$ and $\ell(\gamma) \leq 2 d(x,y)$.
For all $k \in \{ 1,\ldots,n-1 \} $ let 
\[
t_k = \min \{ t \in [0,T] : \ell(\gamma|_{[0,t]}) = \frac{k \, \ell(\gamma)}{n} \} 
,  \]
which exists by continuity.
Set $x_k = \gamma(t_k)$.
Further define $x_0 = x$ and $x_n = y$.
Then 
\[
|x_{k+1} - x_k| 
\leq d(x_{k+1},x_k)
\leq \frac{\ell(\gamma)}{n}
\leq \frac{2 d(x,y)}{n}
\leq \frac{2 c}{n} \, |x-y|
\]
for all $k \in \{ 0,\ldots,n-1 \} $, as required.
\end{proof}

\subsection*{Acknowledgements}
The authors wish to thank Wolfgang Arendt for many discussions at various stages
of this project.
The authors wish to thank Mourad Choulli for helpful comments
regarding the chain condition.

The first-named author is most grateful for the hospitality extended
to him during a fruitful stay at the University of Ulm.
He wishes to thank the University of Ulm for financial support. 
Part of this work is supported by an
NZ-EU IRSES counterpart fund and the Marsden Fund Council from Government funding,
administered by the Royal Society of New Zealand.
Part of this work is supported by the
EU Marie Curie IRSES program, project `AOS', No.~318910.


\begin{thebibliography}{AEKS14}

\bibitem[AE1]{AE1}
{\sc Arendt, W. {\rm and} Elst, A.F.M. ter}, Gaussian estimates for second
  order elliptic operators with boundary conditions.
\newblock {\em J. Operator Theory} {\bf 38} (1997),  87--130.

\bibitem[AE2]{AE2}
\leavevmode\vrule height 2pt depth -1.6pt width 23pt, Sectorial forms and
  degenerate differential operators.
\newblock {\em J. Operator Theory} {\bf 67} (2012),  33--72.

\bibitem[AE3]{AE9}
\leavevmode\vrule height 2pt depth -1.6pt width 23pt, The Dirichlet problem
  without the maximum principle.
\newblock {\em Annales de l'Institut Fourier} {\bf 69} (2019),  763--782.

\bibitem[AEKS]{AEKS}
{\sc Arendt, W., Elst, A.F.M. ter, Kennedy, J.~B. {\rm and} Sauter, M.}, The
  Dirichlet-to-Neumann operator via hidden compactness.
\newblock {\em J. Funct. Anal.} {\bf 266} (2014),  1757--1786.

\bibitem[Aro]{Aro}
{\sc Aronson, D.~G.}, Bounds for the fundamental solution of a parabolic
  equation.
\newblock {\em Bull.\ Amer.\ Math.\ Soc.} {\bf 73} (1967),  890--896.

\bibitem[Aus]{Aus1}
{\sc Auscher, P.}, Regularity theorems and heat kernels for elliptic operators.
\newblock {\em J. London Math.\ Soc.} {\bf 54} (1996),  284--296.

\bibitem[AT]{AT4}
{\sc Auscher, P. {\rm and} Tchamitchian, P.}, Gaussian estimates for second
  order elliptic divergence operators on Lipschitz and $C^1$ domains.
\newblock In {\em Evolution equations and their applications in physical and
  life sciences (Bad Herrenalb, 1998)}, vol.\ 215 of Lecture Notes in Pure and
  Appl.\ Math.,  15--32. Marcel Dekker, New York, 2001.

\bibitem[BE]{BeE3}
{\sc Behrndt, J. {\rm and} Elst, A. F.~M. ter}, Jordan chains of elliptic
  partial differential operators and Dirichlet-to-Neumann maps.
\newblock {\em J. Spectr. Theory} (2019).
\newblock In press.

\bibitem[CK]{ChoulliKayser}
{\sc Choulli, M. {\rm and} Kayser, L.}, Gaussian lower bound for the Neumann
  Green function of a general parabolic operator.
\newblock {\em Positivity} {\bf 19} (2015),  625--646.

\bibitem[Dan1]{Daners}
{\sc Daners, D.}, Heat kernel estimates for operators with boundary conditions.
\newblock {\em Math.\ Nachr.} {\bf 217} (2000),  13--41.

\bibitem[Dan2]{Daners8}
\leavevmode\vrule height 2pt depth -1.6pt width 23pt, Inverse positivity for
  general Robin problems on Lipschitz domains.
\newblock {\em Arch. Math} {\bf 92} (2009),  57--69.

\bibitem[Dav]{Dav2}
{\sc Davies, E.~B.}, {\em Heat kernels and spectral theory}.
\newblock Cambridge Tracts in Mathematics 92. Cambridge University Press,
  Cambridge etc., 1989.

\bibitem[EO1]{EO2}
{\sc Elst, A.F.M. ter {\rm and} Ouhabaz, E.-M.}, Partial Gaussian bounds for
  degenerate differential operators II.
\newblock {\em Ann. Sc. Norm. Super. Pisa Cl. Sci.} {\bf 14} (2015),  37--81.

\bibitem[EO2]{EO6}
\leavevmode\vrule height 2pt depth -1.6pt width 23pt, Poisson bounds for the
  Dirichlet-to-Neumann operator on a $C^{1+\kappa}$-domain.
\newblock {\em J. Differential Equations} {\bf 267} (2019),  4224--4273.

\bibitem[ERe]{ERe2}
{\sc Elst, A.F.M. ter {\rm and} Rehberg, J.}, H{\"o}lder estimates for
  second-order operators on domains with rough boundary.
\newblock {\em Adv. Diff. Equ.} {\bf 20} (2015),  299--360.

\bibitem[ERo]{ER22}
{\sc Elst, A.F.M. ter {\rm and} Robinson, D.W.}, Local lower bounds on heat
  kernels.
\newblock {\em Positivity} {\bf 2} (1998),  123--151.

\bibitem[Gia]{Gia1}
{\sc Giaquinta, M.}, {\em Multiple integrals in the calculus of variations and
  nonlinear elliptic systems}.
\newblock Annals of Mathematics Studies 105. Princeton University Press,
  Princeton, 1983.

\bibitem[Ne{\v{c}}]{Nec2}
{\sc Ne{\v{c}}as, J.}, {\em Direct methods in the theory of elliptic
  equations}.
\newblock Corrected 2nd printing edition, Springer Monographs in Mathematics.
  Springer-Verlag, Berlin, 2012.

\bibitem[Nit]{Nit4}
{\sc Nittka, R.}, Regularity of solutions of linear second order elliptic and
  parabolic boundary value problems on Lipschitz domains.
\newblock {\em J. Differential Equations} {\bf 251} (2011),  860--880.

\bibitem[Ouh]{Ouh5}
{\sc Ouhabaz, E.-M.}, {\em Analysis of heat equations on domains}, vol.\ 31 of
  London Mathematical Society Monographs Series.
\newblock Princeton University Press, Princeton, NJ, 2005.

\end{thebibliography}
\end{document}